\pdfoutput=1
\RequirePackage{ifpdf}
\ifpdf 
\documentclass[pdftex]{sigma}
\else
\documentclass{sigma}
\fi

\usepackage{amscd, mathrsfs, enumerate, mathtools}
\usepackage[all]{xy}
\usepackage[mathcal]{euscript}

\newcommand{\tensor}[1]{{\mathfrak{#1}}}

\newcommand{\bmx}{\begin{pmatrix}}
\newcommand{\emx}{\end{pmatrix}}

\newcommand{\ul}[1]{[#1]}

\DeclareMathOperator{\tr}{tr} 

 \DeclareMathOperator{\End}{End}
\DeclareMathOperator{\cdet}{cdet}

\DeclareMathOperator{\Span}{span}

\def\D{\mathcal D}

\def\p{\mathfrak{p}}

\def\ha{\mbox{\small $\frac{1}{2}$}}

\def\C{\mathcal{C}}
\def\CC{\mathbb{C}}

\def\ZZ{\mathbb{Z}}

\def\P{\mathcal{P}}
\def\U{\mathcal{U}}

\def\1{\tensor{1}}
\def\2{\tensor{2}}
\def\3{\tensor{3}}
\def\4{\tensor{4}}

\newcommand{\mc}{\mathcal}
\newcommand{\mf}{\mathfrak}
\newcommand{\gl}{\mf{gl}}
\newcommand{\so}{\mf{so}}
\renewcommand{\sp}{\mf{sp}}

\newcommand{\ox}{\otimes}
\newcommand{\into}{\hookrightarrow}
\newcommand{\del}{\partial}
\newcommand{\pic}{\pi}
\newcommand{\pict}{\tilde\pi}
\newcommand{\piq}{\hat\pi}
\newcommand{\piqt}{\hat{\tilde\pi}}
\newcommand{\trp}{\,{}^t\!}
\newcommand{\E}{\mathsf E}

\numberwithin{equation}{section}
\newtheorem{Theorem}{Theorem}[section]
\newtheorem{Lemma}[Theorem]{Lemma}
\newtheorem{Proposition}[Theorem]{Proposition}
 { \theoremstyle{definition}
\newtheorem{Definition}[Theorem]{Definition}

\newtheorem{Remark}[Theorem]{Remark} }

\begin{document}

\allowdisplaybreaks

\newcommand{\arXivNumber}{1710.08672}

\renewcommand{\PaperNumber}{040}

\FirstPageHeading

\ShortArticleName{$({\mathfrak{gl}}_M, {\mathfrak{gl}}_N)$-Dualities in Gaudin Models with Irregular Singularities}

\ArticleName{$\boldsymbol{({\mathfrak{gl}}_M, {\mathfrak{gl}}_N)}$-Dualities in Gaudin Models\\ with Irregular Singularities}

\Author{Beno\^{\i}t VICEDO~$^\dag$ and Charles YOUNG~$^\ddag$}

\AuthorNameForHeading{B.~Vicedo and C.~Young}

\Address{$^\dag$~Department of Mathematics, University of York, York YO10 5DD, UK}
\EmailD{\href{mailto:benoit.vicedo@gmail.com}{benoit.vicedo@gmail.com}}

\Address{$^\ddag$~School of Physics, Astronomy and Mathematics, University of Hertfordshire,\\
\hphantom{$^\ddag$}~College Lane, Hatfield AL10 9AB, UK}
\EmailD{\href{mailto:c.a.s.young@gmail.com}{c.a.s.young@gmail.com}}

\ArticleDates{Received November 06, 2017, in final form April 27, 2018; Published online May 03, 2018}

\Abstract{We establish $({\mathfrak{gl}}_M, {\mathfrak{gl}}_N)$-dualities between quantum Gaudin models with irre\-gular singularities. Specifically, for any $M, N \in {\mathbb Z}_{\geq 1}$ we consider two Gaudin models: the one associated with the Lie algebra ${\mathfrak{gl}}_M$ which has a double pole at infinity and $N$ poles, counting multiplicities, in the complex plane, and the same model but with the roles of~$M$ and~$N$ interchanged. Both models can be realized in terms of Weyl algebras, i.e., free bosons; we establish that, in this realization, the algebras of integrals of motion of the two models coincide. At the classical level we establish two further generalizations of the duality. First, we show that there is also a~duality for realizations in terms of free fermions. Second, in the bosonic realization we consider the classical cyclotomic Gaudin model associated with the Lie algebra ${\mathfrak{gl}}_M$ and its diagram automorphism, with a double pole at infinity and $2N$ poles, counting multiplicities, in the complex plane. We prove that it is dual to a non-cyclotomic Gaudin model associated with the Lie algebra ${\mathfrak{sp}}_{2N}$, with a double pole at infinity and~$M$ simple poles in the complex plane. In the special case $N=1$ we recover the well-known self-duality in the Neumann model.}

\Keywords{Gaudin models; dualities; irregular singularities}

\Classification{17B80; 81R12; 82B23}

\section{Introduction}
Fix a set of $N$ distinct complex numbers $\{ z_i \}_{i=1}^N \subset \CC$, and an element $\lambda\in \gl_M^*$. The quadratic Hamiltonians of the quantum Gaudin model \cite{GaudinBookFrench,GaudinBook} associated to $\gl_M$ are the following elements of $U(\gl_M)^{\ox N}$:
\begin{gather*} \mc H_i = \sum_{j\neq i} \sum_{a,b=1}^N \frac{\E_{ab}^{(i)} \E_{ba}^{(j)}}{z_i-z_j} + \sum_{a,b=1}^N \lambda(\E_{ab}) \E_{ba}^{(i)}, \end{gather*}
where $\{ \mathsf E_{ab} \}_{a, b = 1}^M$ denote the standard basis of $\mathfrak{gl}_M$ and $\E_{ab}^{(i)}$ means $\E_{ab}$ in the $i$th tensor factor. The $\mc H_i$ belong to a large commutative subalgebra $\mc Z \subset U(\gl_M)^{\otimes N}$ called the Gaudin~\cite{Fopers} or Bethe~\cite{MTV1} subalgebra, for which an explicit set of generators is known \cite{ChervovTalalaev,MTV1,T}.

If the element $\lambda\in \gl_M^*$ is regular semisimple, i.e., if we can choose bases such that $\lambda(\E_{ab}) = \lambda_a \delta_{ab}$ for some distinct numbers $\{\lambda_a\}_{a=1}^M\subset \CC$, then one can also consider the following elements of $U(\gl_N)^{\ox M}$:
\begin{gather*} \widetilde{\mc H}_a = \sum_{b\neq a} \sum_{i,j=1}^M \frac{\tilde \E_{ij}^{(a)} \tilde \E_{ji}^{(b)}}{\lambda_a-\lambda_b} + \sum_{i=1}^M z_i \tilde \E_{ii}^{(a)}, \end{gather*}
where $\big\{\tilde\E_{ij} \big\}_{i,j = 1}^N$ denote the standard basis of $\mathfrak{gl}_N$. They belong to a large commutative subalgebra $\widetilde{\mc Z} \subset U(\gl_N)^{\otimes M}$.

Let $\CC^M$ denote the defining representation of $\gl_M$. Then $\mc Z$ can be represented as a subalgebra of
\begin{gather*} \End\big(\big(\CC^M\big)^{\ox N}\big) \cong \End\big(\CC^{NM}\big) \cong \End\big(\big(\CC^N\big)^{\ox M}\big).\end{gather*}
So can $\widetilde{\mc Z}$. In fact their images in $\End\big(\CC^{NM}\big)$ coincide. This is the $(\gl_M, \gl_N)$-duality for quantum Gaudin models first observed between the quadratic Gaudin Hamiltonians and the dynamical Hamiltonians in~\cite{TL02}, see also~\cite{TV02}. It was later proved in~\cite{MTVcapelli}, see also~\cite{CF}. (Under this realization the Hamiltonians $\widetilde{\mc H}_a\in \widetilde{\mc Z}$ of the dual model coincide with suitably defined \emph{dynamical Hamiltonians}~\cite{FMTV} of the original $\gl_M$ Gaudin model. See \cite{MTV2, MTVcapelli}.) The classical counterpart of this duality goes back to the works of J.~Harnad \cite{Adams:1990mj,Har94}.

In this paper we generalize this $(\gl_M, \gl_N)$-duality in a number of ways, for both the quantum and classical Gaudin models. Let us describe first the main result. Two natural generalizations of the Gaudin model above are to
\begin{enumerate}[(a)]\itemsep=0pt
\item\label{irreg} models in which the quadratic Hamiltonians (and the Lax matrix, see below) have higher order singularities at the marked points $z_i\in\CC$, $i = 1, \ldots, N$. Such models are called Gaudin models \emph{with irregular singularities}.\footnote{The reason for this terminology is that the spectrum of such models is described in terms of opers with irregular singularities; see~\cite{FFT} and also~\cite{VY3}.
Strictly speaking, the term $\lambda(\E_{ab}) \E_{ab}$ in $\mc H_i$ is already an irregular singularity of order 2 at~$\infty$ in the same sense: namely, the opers describing the spectrum have a double pole at~$\infty$. For that reason we refer to a Gaudin model with such terms in the Hamiltonians $\mc H_i$ as having a double pole at infinity.}
\item\label{jordan} models in which $\lambda\in \gl_{M}^*$ is not semisimple, i.e., has non-trivial Jordan blocks.\footnote{Let us note in passing that the case of $\lambda$ semisimple but not regular is very rich; see for example \cite{FFRyb, Rybnikov,Rcactus}.}
\end{enumerate}
We show that these two generalizations are natural $(\gl_M, \gl_N)$-duals to one another. Namely, we show that there is a correspondence among models generalized in both directions, (\ref{irreg})~and~(\ref{jordan}), and that under this correspondence the sizes of the Jordan blocks get exchanged with the degrees of the irregular singularities at the marked points in the complex plane. See Theorem~\ref{thm: q bispec bos} below.

The heart of the proof is the observation that the generating functions for the generators of both algebras $\mc Z$ and $\widetilde{\mc Z}$ can be obtained by evaluating, in two different ways, the column-ordered determinant of a certain Manin matrix. (A~similar trick was also used in \cite[Proposition~8]{CF}.) Given that observation, the duality between~(\ref{irreg}) and~(\ref{jordan}) above is essentially a consequence of the simple fact that the inverse of a Jordan block matrix
\begin{gather*} \bmx
x & 0 & \ldots & 0\\
-1 & x & \ldots & 0\\
\vdots & \ddots & \ddots & \vdots\\
0 & \ldots & -1 & x
\emx
\qquad\text{is of the form}\qquad
\left(
\begin{matrix}
x^{-1} & 0 & \ldots & 0\\
x^{-2} & x^{-1} & \ldots & 0\\
\vdots & \ddots & \ddots & \vdots\\
x^{-k} & \ldots & x^{-2} & x^{-1}
\end{matrix}
\right);
\end{gather*}
here the higher-order poles in $x$ will give rise to the irregular singularities of the dual Gaudin model.

Now let us give an overview of the results of the paper in more detail. Consider the direct sum of Lie algebras
\begin{gather} \label{Lie alg intro}
\gl_M^{(N)} \coloneqq \bigoplus_{i=1}^N \gl_M \oplus \gl_M^{\rm com},
\end{gather}
where the Lie algebra $\gl_M^{\rm com}$ in the last summand is isomorphic to $\gl_M$ as a vector space but endowed with the trivial Lie bracket. Henceforth we denote the copy of $\mathsf E_{ab}$ in the $i^{\rm th}$ direct summand of $\gl_M^{(N)}$ by $\mathsf E_{ab}^{(z_i)}$ and the copy in the last abelian summand $\gl_M^{\rm com}$ by $\mathsf E_{ab}^{(\infty)}$. In terms of these data, the formal Lax matrix of the Gaudin model associated with $\gl_M$, with a double pole at infinity and simple poles at each $z_i$, $i=1,\ldots,N$, is given by
\begin{gather} \label{Lax intro}
\mathcal L(z) dz \coloneqq \sum_{a,b = 1}^M E_{ba} \otimes \left( \mathsf E_{ab}^{(\infty)} + \sum_{i=1}^N \frac{\mathsf E^{(z_i)}_{ab}}{z - z_i} \right) dz.
\end{gather}
Here $E_{ab} \coloneqq \rho(\mathsf E_{ab})$ where $\rho \colon \gl_M \to \operatorname{Mat}_{M \times M}(\CC)$ is the defining representation.

Regarding $\mathcal L(z)$ as an $M \times M$ matrix with entries in the symmetric algebra $S\big( \gl_M^{(N)} \big)$, the coefficients of its characteristic polynomial
\begin{gather*}
\det\big( \lambda {\bf 1}_{M \times M} - \mathcal L(z) \big)
\end{gather*}
span a large Poisson commutative subalgebra $\mathscr Z^{\rm cl}_{(z_i)}\big( \gl_M^{(N)} \big)$ of $S\big( \gl_M^{(N)} \big)$. Given a classical model described by a Poisson algebra $\P$ and Hamiltonian $H \in \P$, the latter becomes of particular interest if we have a homomorphism of Poisson algebras $\pi\colon S\big( \gl_M^{(N)} \big) \to \P$ such that $H$ lies in the image of $\mathscr Z^{\rm cl}_{(z_i)}\big( \gl_M^{(N)} \big)$. Indeed, $\pi \big( \mathscr Z^{\rm cl}_{(z_i)}\big( \gl_M^{(N)} \big) \big) \subset \P$ then consists of Poisson commuting integrals of motion of the model.

The Lax matrix \eqref{Lax intro} can also be used to describe quantum models by regarding it instead as an $M \times M$ matrix with entries in the universal enveloping algebra $U\big( \gl_M^{(N)} \big)$. In this case, a large commutative subalgebra $\mathscr Z_{(z_i)}\big( \gl_M^{(N)} \big) \subset U\big( \gl_M^{(N)} \big)$, called the \emph{Gaudin algebra}, is spanned by the coefficients in the partial fraction decomposition of the rational functions obtained as the coefficients of the differential operator
\begin{gather*} 
\cdet\big( \partial_z {\bf 1}_{M \times M} - \trp \mathcal L(z) \big),
\end{gather*}
where $\cdet$ is the column ordered determinant.
Given a unital associative algebra $\mathcal U$ and a homomorphism $\piq \colon U\big( \gl_M^{(N)} \big) \to \mathcal U$, the image of $\mathscr Z_{(z_i)}\big( \gl_M^{(N)} \big)$ provides a large commutative subalgebra of~$\mathcal U$.

Let $\mathcal U$ be the Weyl algebra generated by the commuting variables $x^a_i$ for $i=1, \ldots, N$ and $a=1, \ldots, M$ together with their partial derivatives $\partial^a_i \coloneqq \partial / \partial x^a_i$. We introduce another set $\{ \lambda_a \}_{a=1}^M \subset \CC$ of $M$ distinct complex numbers. It is well known that
\begin{gather} \label{piq rep intro}
\piq\big( \mathsf E_{ab}^{(\infty)} \big) = \lambda_a \delta_{ab}, \qquad \piq\big( \mathsf E_{ab}^{(z_i)} \big) = x^a_i \partial^b_i
\end{gather}
defines a homomorphism $\piq \colon U\big( \gl_M^{(N)} \big) \to \mathcal U$. Therefore, in particular, $\piq \big( \mathscr Z_{(z_i)}\big( \gl_M^{(N)} \big) \big)$ is a commutative subalgebra of $\mathcal U$. On the other hand, given the new set of complex numbers $\lambda_a$, $a =1, \ldots, M$, we may now equally consider the Gaudin model associated with $\gl_N$, with a~double pole at infinity and simple poles at each $\lambda_a$ for $a=1,\ldots,M$. Its formal Lax matrix is defined as in~\eqref{Lax intro}, explicitly we let
\begin{gather*}
\tilde{\mathcal L}(\lambda) d\lambda \coloneqq \sum_{i,j = 1}^N \tilde E_{ji} \otimes \left( \tilde{\mathsf E}_{ij}^{(\infty)} + \sum_{a=1}^M \frac{\tilde{\mathsf E}^{(\lambda_a)}_{ij}}{\lambda - \lambda_a} \right) d\lambda.
\end{gather*}
We can define another homomorphism $\piqt \colon U\big( \gl_N^{(M)} \big) \to \mathcal U$ as
\begin{gather*}
\piqt\big( \tilde{\mathsf E}_{ij}^{(\infty)} \big) = z_i \delta_{ij}, \qquad \piqt\big( \tilde{\mathsf E}_{ij}^{(\lambda_a)} \big) = \partial^a_j x^a_i.
\end{gather*}
(Note here the order between $\partial^a_j$ and $x^a_i$ as compared, for instance, to \cite[Section~5.1]{MTV2} where $\tilde{\mathsf E}_{ij}^{(\lambda_a)}$ is realised as $x^a_i \partial^a_j$.) The $(\gl_M, \gl_N)$-duality between the above two Gaudin models associated with $\gl_M$ and $\gl_N$ can be formulated, in the present conventions, as the equality of differential polynomials
\begin{gather*}
\piq \left( \prod_{i=1}^N (z - z_i) \cdet\big( \partial_z {\bf 1}_{M \times M} - \trp \mathcal L(z) \big) \right) = \piqt \left( \prod_{a=1}^M (\partial_z - \lambda_a) \cdet\big( z {\bf 1}_{N \times N} - \tilde{\mathcal L}(\partial_z) \big) \right),
\end{gather*}
whose coefficients are $\mathcal U$-valued polynomials in $z$. (See Section~\ref{sec: quantum Gaudin bispec} for the precise definition of the expression appearing on the right hand side.) In the classical setting discussed above the same identity holds with $\partial_z$ replaced everywhere by the spectral parameter~$\lambda$, the Weyl algebra $\mathcal U$ is replaced by the Poisson algebra $\mathcal P$ defined as the polynomial algebra in the canonically conjugate variables $(p^a_i, x^a_i)$ and column ordered determinants replaced by ordinary determinants.

We generalise this statement in a number of directions. Firstly, in both the classical and quantum cases, we consider Gaudin models with irregular singularities. Specifically, fix a positive integer $n \in \ZZ_{\geq 1}$ and let $\{ \tau_i \}_{i =1}^n \subset \ZZ_{\geq 1}$ be such that $\sum\limits_{i=1}^n \tau_i = N$. We consider a $\gl_M$-Gaudin model with a double pole at infinity and an irregular singularity of order~$\tau_i$ at each~$z_i$ for $i = 1, \ldots, n$. The direct sum of Lie algebras~\eqref{Lie alg intro} is replaced in this case by a direct sum of Takiff Lie algebras\footnote{These were introduced in the mathematics literature in~\cite{Tak71} but have also been widely used in the mathematical physics literature though not by this name, see for instance~\cite{RSTS79}.}
\begin{gather} \label{Takiff intro}
\gl_M^\D \coloneqq \bigoplus_{i=1}^n \big( \gl_M[\varepsilon] / \varepsilon^{\tau_i}\gl_M[\varepsilon] \big) \oplus \gl_M^{\rm com},
\end{gather}
where $\D$ is a divisor encoding the collection of points $z_i$ for $i=1, \ldots, n$ weighted by the integers~$\tau_i$ for $i=1, \ldots, n$. The formal Lax matrix $\mathcal L(z)$ of this Gaudin model is an $M \times M$ matrix with entries in the Lie algebra~$\gl_M^\D$, and the Gaudin algebra $\mathscr Z_{(z_i)}(\gl_M^\D)$ is spanned by the coefficients in the partial fraction decomposition of the rational functions obtained as the coefficients of the differential operator
\begin{gather} \label{cdet intro}
\cdet\big( \partial_z {\bf 1}_{M \times M} - \trp \mathcal L(z) \big).
\end{gather}

Let $\mathcal U$ be the same unital associative algebra as above. In order to define a suitable homomorphism $\piq \colon U\big( \gl_M^\D \big) \to \mathcal U$ we combine representations of the Takiff Lie algebras $\gl_M[\varepsilon] / \varepsilon^{\tau_i}\gl_M[\varepsilon] \to \mathcal U$ for each $i = 1, \ldots, n$, naturally generalising the representation~$\gl_M \to \mathcal U$, $\mathsf E_{ab} \mapsto x^a_i \partial^b_i$ in the above regular singularity case, together with a constant homomorphism $\gl_M^{\rm com} \to \CC 1 \subset \mathcal U$. As before, the choice of the latter is what determines the position of the poles of the dual $\gl_N$-Gaudin model. In fact, if instead of choosing a diagonal matrix as in~\eqref{piq rep intro} we let
\begin{gather*}
\big( \piq\big( \mathsf E_{ab}^{(\infty)} \big) \big)_{a, b=1}^M =
\left(
\begin{matrix}
\lambda_1 & & & & & & & &\\
1 & \lambda_1 & & & & & & 0 &\\
& \ddots & \ddots & & & & & &\\
& & 1 & \lambda_1 & & & & &\\
& & & & \ddots & & &\\
& & & & & \lambda_m & & &\\
& & & & & 1 & \lambda_m & &\\
& 0 & & & & & \ddots & \ddots &\\
& & & & & & & 1 & \lambda_m
\end{matrix}
\right)
\end{gather*}
be a direct sum of $m$ Jordan blocks of size $\tilde{\tau}_a \in \ZZ_{\geq 1}$ with $\lambda_a \in \CC$ along the diagonal for $a = 1, \ldots, m$, such that $\sum\limits_{a=1}^m \tilde{\tau}_a = M$, then the dual Gaudin model associated with $\gl_N$ will have a~double pole at infinity and an irregular singularity at each $\lambda_a$ of order $\tilde{\tau}_a$ for $a = 1, \ldots, m$. Let~$\tilde\D$ be the divisor corresponding to these data and $\gl_N^{\tilde\D}$ the associated direct sum of Takiff algebras, cf.~\eqref{Takiff intro}. After defining a corresponding homomorphism $\piqt \colon U\big( \gl_N^{\tilde \D} \big) \to \mathcal U$ for this Gaudin model, we prove a $(\gl_M, \gl_N)$-duality similar to the one stated above for the regular singularity case, see Theorem~\ref{thm: q bispec bos}. As before, a similar result also holds in the classical setting where~$\pic$ and~$\pict$ in this case are homomorphisms from the symmetric algebras $S\big( \gl_M^\D \big)$ and $S\big( \gl_N^{\tilde\D} \big)$, respectively, to the Poisson algebra~$\mathcal P$, see Theorem~\ref{thm: cla bispec bos}.

In the classical setup of Section~\ref{sec: c} we also consider fermionic generalisations of $(\gl_M, \gl_N)$-duality. Specifically, for the Poisson algebra $\mathcal P$ we take instead the even part of the $\ZZ_2$-graded Poisson algebra generated by canonically conjugate Grassmann variable pairs $(\pi^a_i, \psi^a_i)$. The corresponding homomorphisms of Poisson algebras $\pic_{\mathsf f} \colon S\big( \gl_M^\D \big) \to \mathcal P$ and $\pict_{\mathsf f} \colon S\big( \gl_N^{\tilde\D} \big) \to \mathcal P$ are defined in Lemma~\ref{lem: pi tilde pi ferm}. In this case we establish a different type of $(\gl_M, \gl_N)$-duality between the same Gaudin models with irregular singularities and associated with $\gl_M$ and $\gl_N$ as above. Denoting by $\mathcal L(z)$ and $\tilde{\mathcal L}(\lambda)$ their respective Lax matrices, it takes the form
\begin{gather*}
\pic_{\mathsf f}\big( \det\big( \lambda {\bf 1}_{M \times M} - \mc L(z) \big) \big) \pict_{\mathsf f}\big( \det\big( z {\bf 1}_{N \times N} - \tilde{\mc L}(\lambda) \big) \big) = \prod_{i=1}^n (z - z_i)^{\tau_i} \prod_{a=1}^m (\lambda - \lambda_a)^{\tilde\tau_a}.
\end{gather*}
See Theorem \ref{thm: cla bispec fer}, the proof of which is completely analogous to that of Theorem~\ref{thm: cla bispec bos} in the bosonic setting, using basic properties of the Berezinian of an $(M|N) \times (M|N)$ supermatrix. We leave the possible generalisation of such a fermionic $(\gl_M, \gl_N)$-duality to the quantum setting for future work.

Finally, in Section~\ref{sec: cyclo bispec} we consider extensions of these results to cyclotomic Gaudin models also in the classical setting. Specifically, we consider a $\ZZ_2$-cyclotomic $\gl_M$-Gaudin model with a double pole at infinity as usual and with irregular singularities at the origin of order $\tau_0$ and at points $z_i \in \CC^\times$, with disjoint orbits under $z \mapsto - z$, of order $\tau_i$ for each $i = 1,\dots, n$. Let $N = \tau_0 + \sum\limits_{i=1}^n \tau_i$. Using the bosonic Poisson algebra $\mathcal P$ generated by canonically conjugate variables $(p^a_i, x^a_i)$ we prove that this model is dual to a Gaudin model associated with the Lie algebra $\sp_{2N}$, with a double pole at infinity and regular singularities at $M$ points $\lambda_a$, $a = 1, \ldots, M$, see Theorem~\ref{thm: cla bispec bos cycl}. We show that the well know self-duality in the Neumann model is a particular example of the latter with $N=1$. Generalisations of such $(\gl_M, \gl_N)$-dualities involving cyclotomic Gaudin models to the quantum case are less obvious since it is known~\cite{VY1} that in this case the cyclotomic Gaudin algebra is not generated by a cdet-type formula as in~\eqref{cdet intro}, see Remark~\ref{rem: cyclo quantum}.

\section{Gaudin models with irregular singularities}

\subsection[Lie algebras $\mathfrak{gl}^{\D}_M$ and $\mathfrak{gl}^{\tilde{\D}}_N$]{Lie algebras $\boldsymbol{\mathfrak{gl}^{\D}_M}$ and $\boldsymbol{\mathfrak{gl}^{\tilde{\D}}_N}$}\label{sec: glND}

Let $M, N \in \ZZ_{\geq 1}$. Denote by $\mathsf E_{ab}$ for $a, b = 1, \ldots, M$ the standard basis of $\mathfrak{gl}_M$ and by $\tilde{\mathsf E}_{ij}$ for $i, j = 1, \ldots, N$ the standard basis of $\mathfrak{gl}_N$.

Let $z_i \in \CC$ for $i =1, \ldots, n$ and $\lambda_a \in \CC$ for $a = 1, \ldots, m$ be such that $z_i \neq z_j$ for $i \neq j$ and $\lambda_a \neq \lambda_b$ for $a \neq b$. Pick and fix integers $\tau_i \in \ZZ_{\geq 1}$ for each $i = 1, \ldots, n$ and $\tilde \tau_a \in \ZZ_{\geq 1}$ for each $a = 1, \ldots, m$. We call these the \emph{Takiff degrees} at $z_i$ and $\lambda_a$, respectively.
Consider the effective divisors
\begin{gather*}
\D = \sum_{i=1}^n \tau_i \cdot z_i + 2 \cdot \infty, \qquad \tilde\D = \sum_{a=1}^m \tilde \tau_a \cdot \lambda_a + 2 \cdot \infty.
\end{gather*}
(Recall that an \emph{effective divisor} is a finite formal linear combination of points in some Riemann surface, here the Riemann sphere $\CC\cup \{\infty\}$, with coefficients in $\ZZ_{\geq 0}$.)

We require that $\deg \D = N + 2$ and $\deg \tilde\D = M + 2$ or in other words,
\begin{gather*}
\sum_{i=1}^n \tau_i = N\qquad\text{and}\qquad\sum_{a = 1}^m \tilde \tau_a = M.
\end{gather*}
Note that if $\tau_i = 1 = \tilde \tau_a$ for all $i = 1, \ldots, n$ and $a = 1, \ldots, m$ then in fact we have $n = N$ and $m = M$. More generally, it will be convenient to break up the list of integers from $1$ to~$N$ into~$n$ blocks of sizes~$\tau_i$, $i = 1, \ldots, n$, and similarly for the list of integers from $1$ to $M$. To that end, let us define
\begin{gather} \nu_i := \sum_{j=1}^{i-1} \tau_j,\qquad\text{and}\qquad
 \tilde\nu_a := \sum_{b=1}^{a-1} \tilde\tau_b\label{nudef}\end{gather}
for $i=1,\dots,N$ and $a=1,\dots,M$, so that
\begin{gather*}
(1, \ldots, N) = (1, \dots, \tau_1; \nu_2 +1 , \dots, \nu_2+\tau_2; \dots; \nu_{n}+1, \dots, \nu_{n}+\tau_{n}),\\
(1, \ldots, M) = (1, \dots, \tilde\tau_1; \tilde\nu_2 +1 , \dots, \tilde\nu_2+\tilde\tau_2; \dots;\tilde\nu_{m}+1, \dots, \tilde\nu_{m}+\tilde\tau_{m}).
\end{gather*}
Note that $\nu_1 = \tilde \nu_1 = 0$.

Let $\mathfrak{gl}_M[\varepsilon] \coloneqq \mathfrak{gl}_M \otimes \CC[\varepsilon]$ denote the Lie algebra of polynomials in a formal variable $\varepsilon$ with coefficients in $\gl_M$. For any $k \in \ZZ_{\geq 1}$ we have the ideal $\varepsilon^k\mathfrak{gl}_M[\varepsilon] \coloneqq \mathfrak{gl}_M \otimes \varepsilon^k \CC[\varepsilon]$. The corresponding quotient $\mathfrak{gl}_M[\varepsilon]/\varepsilon^k \coloneqq \mathfrak{gl}_M[\varepsilon]/\varepsilon^k\mathfrak{gl}_M[\varepsilon]$ is called a \emph{Takiff} Lie algebra over $\gl_M$. When $k \in \ZZ_{\geq 2}$, for every $n \in \ZZ_{\geq 1}$ with $n < k$ we have a non-trivial ideal in $\mathfrak{gl}_M[\varepsilon]/\varepsilon^k$ given by $\varepsilon^n \mathfrak{gl}_M[\varepsilon]/\varepsilon^k \coloneqq \varepsilon^n \mathfrak{gl}_M[\varepsilon]/\varepsilon^k\mathfrak{gl}_M[\varepsilon]$, which by abuse of terminology we shall also refer to as a~Ta\-kiff Lie algebra. We define direct sums of Takiff Lie algebras over $\mathfrak{gl}_M$ and $\mathfrak{gl}_N$, respectively, as
\begin{gather*}
\mathfrak{gl}_M^\D \coloneqq \varepsilon_\infty \mathfrak{gl}_M[\varepsilon_\infty] / \varepsilon_\infty^2 \oplus \bigoplus_{i=1}^n \mathfrak{gl}_M[\varepsilon_{z_i}] / \varepsilon_{z_i}^{\tau_i},\\
\mathfrak{gl}_N^{\tilde\D} \coloneqq \tilde\varepsilon_\infty \mathfrak{gl}_N[\tilde\varepsilon_\infty] / \tilde\varepsilon_\infty^2 \oplus \bigoplus_{a=1}^m \mathfrak{gl}_N[\tilde\varepsilon_{\lambda_a}] / \tilde\varepsilon_{\lambda_a}^{\tilde \tau_a}.
\end{gather*}
Note that $\varepsilon_\infty \mathfrak{gl}_M[\varepsilon_\infty] / \varepsilon_\infty^2$ and $\tilde\varepsilon_\infty \mathfrak{gl}_N[\tilde\varepsilon_\infty] / \tilde\varepsilon_\infty^2$ are respectively isomorphic to the abelian Lie algebras $\gl^{\rm com}_M$ and $\gl^{\rm com}_N$ in the notation used in the introduction, see, e.g.,~\eqref{Lie alg intro}.

We use the abbreviated notation $\mathsf X \varepsilon^k$ for an element $\mathsf X \otimes \varepsilon^k \in \mathfrak{gl}_M[\varepsilon]$ where $\mathsf X \in \mathfrak{gl}_M$ and $k \in \ZZ_{\geq 0}$, and likewise for elements of $\mathfrak{gl}_N[\varepsilon]$. Fix a basis of $\mathfrak{gl}_M^\D$ defined by
\begin{gather*}
\mathsf E^{(z_i)}_{ab \ul r} \coloneqq \mathsf E_{ab} \varepsilon_{z_i}^r, \qquad \mathsf E^{(\infty)}_{ab \ul 1} \coloneqq \mathsf E_{ab} \varepsilon_\infty
\end{gather*}
for $i = 1, \ldots, N$, $a, b = 1, \ldots, M$ and $r = 0, \ldots, \tau_i - 1$. Let us note, in particular, that $\mathsf E^{(z_i)}_{ab \ul r} = 0$ whenever $r \geq \tau_i$. Likewise, as a basis of $\mathfrak{gl}_N^{\tilde \D}$ we take{\samepage
\begin{gather*}
\tilde{\mathsf E}^{(\lambda_a)}_{ij \ul s} \coloneqq \tilde{\mathsf E}_{ij} \tilde\varepsilon_{\lambda_a}^s, \qquad
\tilde{\mathsf E}^{(\infty)}_{ij \ul 1} \coloneqq \tilde{\mathsf E}_{ij} \tilde\varepsilon_\infty
\end{gather*}
for $a = 1, \ldots, M$, $i, j = 1, \ldots, N$ and $s = 0, \ldots, \tilde \tau_a - 1$. Here also $\tilde{\mathsf E}^{(\lambda_a)}_{ij \ul s} = 0$ for $s \geq \tilde \tau_a$.}

The set of non-trivial Lie brackets of these basis elements read
\begin{gather}\label{M takiff rels}
\big[ \mathsf E^{(z_i)}_{ab \ul r}, \mathsf E^{(z_j)}_{cd \ul s} \big] = \delta_{ij} [\mathsf E_{ab}, \mathsf E_{cd}]^{(z_i)}_{\ul {r+s}} = \delta_{ij} \delta_{bc} \mathsf E^{(z_i)}_{ad \ul {r+s}} - \delta_{ij} \delta_{ad} \mathsf E^{(z_i)}_{cb \ul {r+s}},
\end{gather}
for any $i, j = 1, \ldots, n$ and $a, b, c, d = 1, \ldots, M$, and
\begin{gather*}
\big[ \tilde{\mathsf E}^{(\lambda_a)}_{ij \ul r}, \tilde{\mathsf E}^{(\lambda_b)}_{kl \ul s} \big] = \delta_{ab} [\tilde{\mathsf E}_{ij}, \tilde{\mathsf E}_{kl}]^{(\lambda_a)}_{\ul {r+s}} = \delta_{ab} \delta_{jk} \tilde{\mathsf E}^{(\lambda_a)}_{il \ul {r+s}} - \delta_{ab} \delta_{il} \tilde{\mathsf E}^{(\lambda_a)}_{kj \ul {r+s}},
\end{gather*}
for any $i, j, k, l = 1, \ldots, N$ and $a, b = 1, \ldots, m$. Note, in particular, that $\mathsf E^{(\infty)}_{ab \ul 1}$ and $\tilde{\mathsf E}^{(\infty)}_{ij \ul 1}$ are Casimirs of the Lie algebras $\mathfrak{gl}^\D_M$ and $\mathfrak{gl}^{\tilde\D}_N$, respectively.

\subsection{Lax matrices}
Let $\rho \colon \mathfrak{gl}_M \to \operatorname{Mat}_{M \times M}(\CC)$ and $\tilde\rho \colon \mathfrak{gl}_N \to \operatorname{Mat}_{N \times N}(\CC)$ denote the defining representations of~$\mathfrak{gl}_M$ and~$\mathfrak{gl}_N$, respectively. We write $E_{ab} \coloneqq \rho(\mathsf E_{ab})$ and $\tilde E_{ij} \coloneqq \tilde\rho\big(\tilde{\mathsf E}_{ij}\big)$.

The sets $\{ \mathsf E_{ab} \}_{a,b=1}^M$ and $\{ \mathsf E_{ba} \}_{a,b =1}^M$ form dual bases of $\mathfrak{gl}_M$ with respect to the trace in the representation $\rho$ since $\tr(E_{ab} E_{cd}) = \delta_{ad} \delta_{bc}$ for all $a, b, c, d = 1, \ldots, M$. Likewise, dual bases of~$\mathfrak{gl}_N$ with respect to the trace in the representation $\tilde \rho$ are given by $\big\{ \tilde{\mathsf E}_{ij} \big\}_{i,j=1}^N$ and $\big\{ \tilde{\mathsf E}_{ji} \big\}_{i,j=1}^N$.

The Lax matrix of the Gaudin model associated with $\gl^\D_M$ is given by
\begin{subequations} \label{formal Lax glM glN Takiff}
\begin{gather}\label{formal Lax glM}
\mc L^\D(z) dz \coloneqq \sum_{a, b=1}^M E_{ba} \otimes \left( \mathsf E^{(\infty)}_{ab \ul 1} + \sum_{i=1}^n \sum_{r=0}^{\tau_i - 1} \frac{\mathsf E^{(z_i)}_{ab \ul r}}{(z - z_i)^{r+1}} \right) dz.
\end{gather}
It is an $M\times M$ matrix whose coefficients are rational functions of $z$ valued in $\gl_M^\D$. Likewise, the Lax matrix of the Gaudin model associated with $\gl^{\tilde \D}_N$ reads
\begin{gather}
\label{formal Lax glN} \mc L^{\tilde \D}(\lambda) d\lambda \coloneqq \sum_{i, j=1}^N \tilde{E}_{ji} \otimes \left( \tilde{\mathsf E}^{(\infty)}_{ij \ul 1} + \sum_{a=1}^m \sum_{s=0}^{\tilde \tau_a - 1} \frac{\tilde{\mathsf E}^{(\lambda_a)}_{ij \ul s}}{(\lambda - \lambda_a)^{s+1}} \right) d\lambda,
\end{gather}
\end{subequations}
and is an $N\times N$ matrix with entries rational functions of $\lambda$ valued in $\gl_N^{\tilde \D}$.

\section[Classical $(\gl_M, \gl_N)$-duality]{Classical $\boldsymbol{(\gl_M, \gl_N)}$-duality} \label{sec: c}

\subsection{Classical Gaudin model} \label{sec: spectral curves}

The algebra of observables of the classical Gaudin model associated with $\gl^\D_M$ is the symmetric tensor algebra $S\big(\gl^\D_M\big)$. It is a Poisson algebra: the Poisson bracket is defined to be equal to the Lie bracket \eqref{M takiff rels} on the subspace $\gl^\D_M\into S\big(\gl^\D_M\big)$ and then extended by the Leibniz rule to the whole of $S\big(\gl^\D_M\big)$. Consider the quantity
\begin{gather} \label{spectral curve glM}
\prod_{i=1}^n (z - z_i)^{\tau_i} \det\big(\lambda \mathbf 1_{M\times M} - \mc L^\D(z)\big).
\end{gather}
This is a polynomial of degree $M$ in $\lambda$ whose coefficients are rational functions in $z$ with coefficients in $S\big(\gl^\D_M\big)$. The \emph{classical Gaudin algebra $\mathscr Z^{\rm cl}\big(\gl^\D_M\big)$ } of the $\gl^\D_M$-Gaudin model is by definition the linear subspace of $S\big(\gl^\D_M\big)$ spanned by these coefficients. It is a Poisson-commutative subalgebra of $S\big(\gl^\D_M\big)$.

The classical Gaudin algebra $\mathscr Z^{\rm cl}\big(\gl^{\tilde \D}_N\big)$ of the $\gl^{\tilde \D}_N$-Gaudin model is defined analogously in terms of the following polynomial of degree $N$ in $z$ with coefficients rational in $\lambda$,
\begin{gather} \label{spectral curve glN}
\prod_{a=1}^m (\lambda - \lambda_a)^{\tilde\tau_a} \det \big(z \mathbf 1_{N\times N} - \mc L^{\tilde \D}(\lambda)\big).
\end{gather}

\subsection{Bosonic realisation} \label{sec: bos real cla}
Introduce the Poisson algebra $\P_{\mathsf b} \coloneqq \CC\big[x^a_i, p^b_j\big]_{i,j=1\; a,b=1}^{N\quad\;\, M}$ with Poisson brackets
\begin{gather} \label{PB for Pb}
\big\{ x^a_i, x^b_j \big\} = 0, \qquad \big\{ p^a_i, x^b_j \big\} = \delta_{ij} \delta_{ab}, \qquad \big\{ p^a_i, p^b_j \big\} = 0,
\end{gather}
for $a, b = 1, \ldots, M$ and $i, j = 1, \ldots, N$. In the following we shall regard $\P_{\mathsf b}$ as a Lie algebra under the Poisson bracket.

For any $x \in \CC$ and $k \in \ZZ_{\geq 1}$ we denote by $J_k(x)$ the Jordan block of size $k \times k$ with $x$ along the diagonal and $-1$'s below the diagonal, namely
\begin{gather*} 
J_k(x) = \left(
\begin{matrix}
x & 0 & \ldots & 0\\
-1 & x & \ldots & 0\\
\vdots & \ddots & \ddots & \vdots\\
0 & \ldots & -1 & x
\end{matrix}
\right).
\end{gather*}
We note for later that if $x \neq 0$ then this is invertible and its inverse is given by
\begin{gather} \label{J block inv}
J_k(x)^{-1} = \left(
\begin{matrix}
x^{-1} & 0 & \ldots & 0\\
x^{-2} & x^{-1} & \ldots & 0\\
\vdots & \ddots & \ddots & \vdots\\
x^{-k} & \ldots & x^{-2} & x^{-1}
\end{matrix}
\right).
\end{gather}

\begin{Lemma} \label{lem: pic pict}
The linear maps $\pic_{\mathsf b}\colon \mathfrak{gl}_M^\D \to \P_{\mathsf b}$ and $\pict_{\mathsf b}\colon \mathfrak{gl}_N^{\tilde\D} \to \P_{\mathsf b}$
defined by
\begin{gather*}
\pic_{\mathsf b}\big(\mathsf E^{(z_i)}_{ab \ul r}\big) =
\sum_{u= \nu_i+1}^{\nu_i +\tau_i - r}
x^a_{u+r} p^b_u, \qquad
\pic_{\mathsf b}\big(\mathsf E^{(\infty)}_{ab \ul 1}\big) = - \left( \bigoplus_{c=1}^m J_{\tilde \tau_c}(-\lambda_c) \right)_{ba},
\end{gather*}
for every $r = 0, \ldots, \tau_i -1$, $i = 1, \ldots, n$ and $a, b = 1,\ldots, M$, and
\begin{gather*}
\pict_{\mathsf b}\big(\tilde{\mathsf E}^{(\lambda_a)}_{ij \ul s} \big) = \sum_{u= \tilde\nu_a+1}^{\tilde\nu_a + \tilde\tau_a - s}p^u_j x^{u+s}_i, \qquad
\pict_{\mathsf b}\big(\tilde{\mathsf E}^{(\infty)}_{ij \ul 1} \big) = - \left( \bigoplus_{k=1}^n J_{\tau_k}(-z_k) \right)_{ji},
\end{gather*}
for every $s = 0, \ldots, \tilde\tau_a - 1$, $i,j = 1, \ldots, N$ and $a = 1,\ldots, m$, are homomorphisms of Lie algebras. They extend uniquely to homomorphisms of Poisson algebras $\pic_{\mathsf b}\colon S\big(\mathfrak{gl}_M^\D\big) \to \P_{\mathsf b}$ and $\pict_{\mathsf b}\colon S\big(\mathfrak{gl}_N^{\tilde\D}\big) \to \P_{\mathsf b}$.
\end{Lemma}
\begin{proof} We will prove the corresponding result in the quantum case in detail below. See Lemma~\ref{lem: pi tilde pi quant b}. That proof applies line-by-line here, with $\del$ replaced by $p$.
\end{proof}

Let $\CC(z)[\lambda]$ denote the algebra of polynomials in $\lambda$ with coefficients rational in $z$. Given any Poisson algebra $\mathcal P$ we introduce the Poisson algebra $\mathcal P(z)[\lambda] \coloneqq \mathcal P \otimes \CC(z)[\lambda]$ with Poisson bracket defined using multiplication in the second tensor factor.
Extend the homomorphisms $\pic_{\mathsf b}$ and $\pict_{\mathsf b}$ from Lemma \ref{lem: pic pict} to homomorphisms of Poisson algebras
\begin{gather*}
\pic_{\mathsf b}\colon \ S\big(\mathfrak{gl}_M^\D\big)(z)[\lambda] \longrightarrow \P_{\mathsf b}(z)[\lambda], \qquad
\pict_{\mathsf b}\colon \ S\big(\mathfrak{gl}_M^{\tilde \D}\big)(\lambda)[z] \longrightarrow \P_{\mathsf b}(\lambda)[z],
\end{gather*}
by letting them act trivially on the tensor factors $\CC(z)[\lambda]$ and $\CC(\lambda)[z]$, respectively. In particular, we may apply these homomorphisms respectively to the expressions~\eqref{spectral curve glM} and~\eqref{spectral curve glN}. It follows from Theorem~\ref{thm: cla bispec bos} below that the resulting expressions in fact live in the common subalgebra $\P_{\mathsf b}[z, \lambda] \coloneqq \P_{\mathsf b} \otimes \CC[z, \lambda]$ of both $\P_{\mathsf{b}}(z)[\lambda]$ and $\P_{\mathsf{b}}(\lambda)[z]$, where $\CC[z, \lambda]$ denotes the algebra of polynomials in the variables $z$ and $\lambda$. The coefficients of these polynomials in $\P_{\mathsf{b}}[z, \lambda]$ span the images of the classical Gaudin algebras in $\P_{\mathsf b}$, namely
\begin{gather*} \pic_{\mathsf b}\big(\mathscr Z^{{\rm cl}}\big(\gl^\D_M\big)\big) \subset \P_\mathsf{b}\qquad\text{and}\qquad
\pict_{\mathsf b}\big(\mathscr Z^{\rm cl}\big(\gl^{\tilde \D}_N\big)\big) \subset \P_\mathsf{b},\end{gather*}
respectively. The following theorem establishes that these Poisson-commutative subalgebras of~$\P_{\mathsf b}$ coincide.

\begin{Theorem} \label{thm: cla bispec bos}
We have the following relation
\begin{gather*}
\pic_{\mathsf b}\left( \prod_{i=1}^n (z - z_i)^{\tau_i} \det\big( \lambda {\bf 1}_{M \times M} - \mc L^\D(z) \big) \right) = \pict_{\mathsf b}\left( \prod_{a=1}^m (\lambda - \lambda_a)^{\tilde\tau_a} \det\big( z {\bf 1}_{N \times N} - \mc L^{\tilde\D}(\lambda) \big) \right),
\end{gather*}
as an equality in $\P_{\mathsf b}[z, \lambda]$.
\end{Theorem}
\begin{proof}
Introduce the $M \times M$ and $N \times N$ block diagonal matrices
\begin{gather*}\Lambda \coloneqq \bigoplus_{a = 1}^m \trp J_{\tilde\tau_a}(\lambda - \lambda_a),\qquad
Z \coloneqq \bigoplus_{i = 1}^n J_{\tau_i}(z - z_i).
\end{gather*}
Also introduce the $M \times N$ matrices
\begin{gather*}
P \coloneqq (p^a_i)_{a = 1\;i=1}^{M \;\;\;\, N},
\qquad
X \coloneqq (x^a_i)_{a = 1\;i=1}^{M \;\;\;\, N}.
\end{gather*}
Consider the block matrix
\begin{gather}\label{clasMdef}
M \coloneqq \left(
\begin{matrix}
 \Lambda & X\\
\trp P & Z
\end{matrix}
\right),
\end{gather}
with entries in the commutative algebra $\mc P_{\mathsf b}[\lambda, z]$. We may evaluate its determinant in two ways.
On the one hand, we have
\begin{align*}
\det M&= \det \left( M \bmx 1 & - \Lambda^{-1} X \\ 0 & 1 \emx \right) \\
&= \det \bmx \Lambda & 0 \\ \trp P & Z- \trp P\Lambda^{-1} X\emx= \det \Lambda \det\left( Z- \trp P\Lambda^{-1} X \right). \end{align*}
On the other hand,
\begin{align*}
\det M &= \det \bmx Z & \trp P\\ X & \Lambda \emx
= \det \left( \bmx Z & \trp P\\ X & \Lambda \emx
 \bmx 1 & - Z^{-1} \trp P \\ 0 & 1 \emx \right) \\
&= \det \bmx Z & 0 \\ X & \Lambda- X Z^{-1} \trp P\emx
= \det Z \det\big( \Lambda- X Z^{-1} \trp P \big). \end{align*}
Hence we obtain the relation
\begin{gather} \label{clas bos master eq}
\det Z \det\big(\Lambda - X Z^{-1} \trp P\big) = \det \Lambda \det\big(Z - \trp P \Lambda^{-1} X\big).
\end{gather}
It remains to note that the square matrices $Z$ and $\Lambda$ can be written as
\begin{gather*}
Z = \sum_{i,j=1}^N \tilde E_{ij} \big( z \delta_{ij} - \pi_{\mathsf b}\big(\tilde{\mathsf E}^{(\infty)}_{ij \ul 1}\big) \big), \qquad
\Lambda = \sum_{a,b=1}^M E_{ab} \big( \lambda \delta_{ab} - \pi_{\mathsf b}\big(\mathsf E^{(\infty)}_{ba \ul 1}\big) \big)
\end{gather*}
with $\pi_{\mathsf b}$ and $\tilde\pi_{\mathsf b}$ as defined in Lemma \ref{lem: pic pict}, and that their inverses are given by
\begin{gather*}
Z^{-1} = \bigoplus_{i = 1}^n J_{\tau_i}(z - z_i)^{-1}, \qquad
\Lambda^{-1} = \bigoplus_{a = 1}^m \trp J_{\tilde\tau_a}(\lambda - \lambda_a)^{-1}.
\end{gather*}
Thus we have
\begin{align*}
\Lambda - X Z^{-1} \trp P &= \sum_{a,b = 1}^M E_{ab} \big(\Lambda - X Z^{-1} \trp P\big)_{ab}\\
& = \lambda {\bf 1} - \sum_{a, b= 1}^M E_{ab} \left( \pic_{\mathsf b}\big(\mathsf E^{(\infty)}_{ab \ul 1}\big) + \sum_{i=1}^n \sum_{j, k = \nu_i+1}^{\nu_i+ \tau_i}
x^a_j \big( J_{\tau_i}(z - z_i)^{-1} \big)_{jk} p^b_k \right),
\end{align*}
which is nothing but $\lambda {\bf 1} - \pic_{\mathsf b}\big( \trp \mc L^\D(z) \big)$ using Lemma~\ref{lem: pi tilde pi quant b}, the expression \eqref{formal Lax glM} for the Lax matrix~$\mc L^\D(z)$ and~\eqref{J block inv} for the inverse of a Jordan block. Likewise
\begin{align*}
Z - \trp P \Lambda^{-1} X &= \sum_{i, j = 1}^N \tilde E_{ij} \big(Z - \trp P \Lambda^{-1} X\big)_{ij}\\
& = z {\bf 1} - \sum_{i,j = 1}^N \tilde E_{ij} \left( \pict_{\mathsf b}\big( \tilde {\mathsf E}^{(\infty)}_{ji \ul 1} \big) + \sum_{a = 1}^m
\sum_{b,c = \tilde\nu_a+1}^{\tilde\nu_a+ \tilde\tau_a} p^b_i \big( J_{\tilde\tau_a}(\lambda - \lambda_a)^{-1} \big)_{cb} x^c_j \right),
\end{align*}
which coincides with $z {\bf 1} - \pict_{\mathsf b} \big( \tilde{\mc L}(\lambda) \big)$, as required. Since $\det \trp A = \det A$ for any square matrix $A$ and noting that $\det Z = \prod\limits_{i=1}^n (z - z_i)^{\tau_i}$ and $\det \Lambda = \prod\limits_{a=1}^m (\lambda - \lambda_a)^{\tilde\tau_a}$, the result follows.
\end{proof}

\subsection{Fermionic realisation}

Let $V \coloneqq \Span_\CC \big\{ \psi^a_i, \pi^b_j \big\}_{i,j=1\; a,b=1}^{N\quad\;\, M}$ and define the exterior algebra
$\P_{\mathsf f} \coloneqq \bigwedge V = \bigoplus_{k=0}^{2 M N} \bigwedge^k V$,
whose skew-symmetric product we denote simply by juxtaposition. We refer to an element $u \in \bigwedge^k V$ as being homogeneous of degree $k$ and write $|u| = k$. In particular, $|\psi^a_i| = |\pi^a_i| = 1$ for any $a = 1, \ldots, M$ and $i = 1, \ldots, N$. We endow $\P_{\mathsf f}$ with a $\ZZ_2$-graded Poisson structure defined by
\begin{gather*}
\big\{ \pi^a_i, \psi^b_j \big\}_+ = \big\{ \psi^b_j, \pi^a_i \big\}_+ = \delta_{ij} \delta_{ab},
\end{gather*}
for any $a, b = 1, \ldots, M$ and $i, j = 1, \ldots, N$, and extended to the whole of $\P_{\mathsf f}$ by the $\ZZ_2$-graded skew-symmetry property and the $\ZZ_2$-graded Leibniz rule, i.e.,
\begin{gather*}
\{ u, v \}_+ = - (-1)^{|u| |v|} \{ v, u \}_+,\\
\{ u, v w \}_+ = \{ u, v \}_+ w + (-1)^{|u| |v|} v \{ u, w \}_+
\end{gather*}
for any homogeneous elements $u, v, w \in \P_{\mathsf f}$.

Let $\P_{\mathsf f}^{\bar 0} \coloneqq \bigoplus_{k=0}^{M N} \bigwedge^{2k} V$ denote the even subspace of $\P_{\mathsf f}$. The restriction of the $\ZZ_2$-graded Poisson bracket $\{ \cdot, \cdot \}_+$ to $\P_{\mathsf f}^{\bar 0}$ defines a Lie algebra structure on $\P_{\mathsf f}^{\bar 0}$.

\begin{Lemma} \label{lem: pi tilde pi ferm}
The linear maps $\pic_{\mathsf f}\colon \mathfrak{gl}_M^\D \to \P_{\mathsf f}^{\bar 0}$ and $\pict_{\mathsf f}\colon \mathfrak{gl}_N^{\tilde\D} \to \P_{\mathsf f}^{\bar 0}$
defined by
\begin{gather*}
\pic_{\mathsf f}\big(\mathsf E^{(z_i)}_{ab \ul r}\big) = \sum_{u= \nu_i+1}^{\nu_i +\tau_i - r} \pi^a_{u+r} \psi^b_u, \qquad
\pic_{\mathsf f}\big(\mathsf E^{(\infty)}_{ab \ul 1}\big) = - \left( \bigoplus_{c=1}^m J_{\tilde \tau_c}(-\lambda_c) \right)_{ab},
\end{gather*}
for every $i = 1, \ldots, n$ and $a, b = 1,\ldots, M$, and
\begin{gather*}
\pict_{\mathsf f}\big(\tilde{\mathsf E}^{(\lambda_a)}_{ij \ul s} \big) =
\sum_{u= \tilde\nu_a+1}^{\tilde\nu_a + \tilde\tau_a - s} \psi^u_i \pi^{u+s}_j, \qquad
\pict_{\mathsf f}\big(\tilde{\mathsf E}^{(\infty)}_{ij \ul 1} \big) = - \left( \bigoplus_{k=1}^n J_{\tau_k}(-z_k) \right)_{ij},
\end{gather*}
for every $i,j = 1, \ldots, N$ and $a = 1,\ldots, m$, are homomorphisms of Lie algebras.
\end{Lemma}
\begin{proof}
For each $i,j = 1, \ldots, n$ and $a, b = 1,\ldots, M$ we have
\begin{align*}
\big\{ \pic_{\mathsf f}\big(\mathsf E^{(z_i)}_{ab \ul r}\big), \pic_{\mathsf f}\big(\mathsf E^{(z_j)}_{cd \ul s}\big) \big\}_+ &=
\sum_{u= \nu_i+1}^{\nu_i +\tau_i - r}
\sum_{v= \nu_j+1}^{\nu_j +\tau_j - s} \{ \pi^a_{u+r} \psi^b_u, \pi^c_{v+s} \psi^d_v \}_+\\
& = \sum_{u = \nu_i+1}^{\nu_i+\tau_i - r}
\sum_{v= \nu_i+1}^{\nu_i+\tau_i - s} \big( \pi^a_{u+r} \big\{ \psi^b_u, \pi^c_{v+s} \big\}_+ \psi^d_v - \pi^c_{v+s} \big\{ \pi^a_{u+r}, \psi^d_v \big\}_+ \psi^b_u \big) \delta_{ij}\\
&=\sum_{u = \nu_i+1}^{\nu_i + \tau_i - r - s} \big( \delta_{cb} \pi^a_{u+r+s} \psi^d_u - \delta_{ad} \pi^c_{u+r+s} \psi^b_u \big) \delta_{ij}\\
& = \big( \delta_{bc} \pic_{\mathsf f}\big(\mathsf E^{(z_i)}_{ad \ul {r+s}}\big) - \delta_{ad} \pic_{\mathsf f}\big(\mathsf E^{(z_i)}_{cb \ul {r+s}}\big) \big) \delta_{ij} = \pic_{\mathsf f} \big( \big\{ \mathsf E^{(z_i)}_{ab \ul r}, \mathsf E^{(z_j)}_{cd \ul s} \big\} \big).
\end{align*}
Likewise, for each $i,j = 1, \ldots, N$ and $a, b = 1,\ldots, m$ one shows that
\begin{align*}
&\big\{ \pict_{\mathsf f}\big(\tilde{\mathsf E}^{(\lambda_a)}_{ij \ul r}\big), \pict_{\mathsf f}\big(\tilde{\mathsf E}^{(\lambda_b)}_{kl \ul s}\big) \big\}_+ = \pict_{\mathsf f} \big( \big\{ \tilde{\mathsf E}^{(\lambda_a)}_{ij \ul r}, \tilde{\mathsf E}^{(\lambda_b)}_{kl \ul s} \big\} \big),
\end{align*}
and all Poisson brackets involving the generators at infinity are also easily seen to be preserved by the linear maps $\pic_{\mathsf f}$ and $\pict_{\mathsf f}$ since $z_i \in \CC$ and $\lambda_a \in \CC$ are central in $\P_{\mathsf f}^{\bar 0}$.
\end{proof}

\begin{Theorem} \label{thm: cla bispec fer}
We have the following relation
\begin{gather*}
\pic_{\mathsf f}\big( \det\big( \lambda {\bf 1}_{M \times M} - \mc L^\D(z) \big) \big) \pict_{\mathsf f}\big( \det\big( z {\bf 1}_{N \times N} - \mc L^{\tilde \D}(\lambda) \big) \big) = \prod_{i=1}^n (z - z_i)^{\tau_i} \prod_{a=1}^m (\lambda - \lambda_a)^{\tilde\tau_a}.
\end{gather*}
\end{Theorem}

\begin{proof}
Consider the same $M \times M$ and $N \times N$ block diagonal matrices $Z$ and $\Lambda$ as in the proof of Theorem~\ref{thm: cla bispec bos}. Introduce the $M \times N$ and $N \times M$ matrices
\begin{gather*}
\Pi \coloneqq (\pi^a_i)_{a=1 \; i = 1}^{M \;\;\;\, N}, \qquad \Psi \coloneqq (\psi^a_i)_{i=1 \; a = 1}^{N \;\;\;\, M},
\end{gather*}
and consider the following even supermatrix
\begin{gather*}
M \coloneqq \left(
\begin{array}{c|c}
\Lambda & \Pi\\
\hline
\Psi & Z
\end{array}
\right).
\end{gather*}
Since $\Lambda$ and $Z$ are both invertible, we can define the Berezinian, or superdeterminant, of $M$ which is given by $\operatorname{Ber} M = \det \Lambda \big( \det\big(Z - \Psi \Lambda^{-1} \Pi\big) \big)^{-1}$. Alternatively, the Berezinian of $M$ can equally be expressed as $\operatorname{Ber} M = \det\big(\Lambda - \Pi Z^{-1} \Psi\big) (\det Z)^{-1}$, see for instance~\cite{MR748771}. Equating these two expressions of $\operatorname{Ber} M$ we obtain the relation
\begin{gather*}
\det\big(\Lambda - \Pi Z^{-1} \Psi\big) \det\big(Z - \Psi \Lambda^{-1} \Pi\big) = \det Z \det \Lambda.
\end{gather*}

Recalling the expressions for the square matrices $Z$ and $\Lambda$ and their inverses given in the proof of Theorem~\ref{thm: cla bispec bos}, we can write
\begin{align*}
\Lambda - \Pi Z^{-1} \Psi &= \sum_{a,b = 1}^M E_{ab} \big(\Lambda - \Pi Z^{-1} \Psi\big)_{ab}\\
& = \lambda {\bf 1} - \sum_{a, b= 1}^M E_{ab} \left( \pic_{\mathsf f}\big(\mathsf E^{(\infty)}_{ba \ul 1}\big) + \sum_{i=1}^n
\sum_{j, k = \nu_i+1}^{\nu_i + \tau_i} \pi^a_j \big( J_{\tau_i}(z - z_i)^{-1} \big)_{jk} \psi^b_k \right),
\end{align*}
which is nothing but $\lambda {\bf 1} - \pic_{\mathsf f}\big( \mathcal L^\D(z) \big)$. Likewise
\begin{align*}
Z - \Psi \Lambda^{-1} \Pi & = \sum_{i, j = 1}^N E_{ij} \big(Z - \Psi \Lambda^{-1} \Pi\big)_{ij}\\
& = z {\bf 1} - \sum_{i,j = 1}^N E_{ij} \left( \pict_{\mathsf f}\big( \mathsf E^{(\infty)}_{ij \ul 1} \big) + \sum_{a = 1}^m
\sum_{b,c = \tilde\nu_a+1}^{\tilde\nu_a + \tilde\tau_a} \psi^b_i \big( J_{\tilde\tau_a}(\lambda - \lambda_a)^{-1} \big)_{cb} \pi^c_j \right),
\end{align*}
which is $z {\bf 1} - \pict_{\mathsf f}\big( \trp \mathcal L^{\tilde \D}(\lambda) \big)$.
The result now follows as in the proof of Theorem \ref{thm: cla bispec bos}.
\end{proof}

\section[Quantum $(\gl_M, \gl_N)$-duality]{Quantum $\boldsymbol{(\gl_M, \gl_N)}$-duality} \label{sec: q}

There is a natural quantum version of Theorem~\ref{thm: cla bispec bos}. In order to state it, we first need a short digression on Manin matrices. In this section we do not consider the fermionic counterpart of Theorem~\ref{thm: cla bispec bos}, namely Theorem~\ref{thm: cla bispec fer}, but leave this for future work.

\subsection{Manin matrices}\label{sec: manin matrices}

Let $\mc A$ be an associative (but possibly noncommutative) algebra over $\CC$. Suppose $M=(M_{ij})$ is a matrix with entries in $\mc A$.
\begin{Definition}\label{def: Manin}
The matrix $M$ is a \emph{Manin matrix} if
\begin{enumerate}[(i)]\itemsep=0pt
\item $[M_{ij}, M_{kj}] = 0$ for all $i$, $j$, $k$, and
\item $[M_{ij},M_{kl}] = [M_{kj},M_{il}]$ for all $i$, $j$, $k$, $l$.
\end{enumerate}
\end{Definition}
That is, elements of the same column must commute amongst themselves, and commutators of cross terms of $2\times 2$ submatrices must be equal (for example $[M_{11},M_{22}] = [M_{21},M_{12}]$). Actually the second of these conditions implies the first (set $j=l$) but it is convenient to think of them separately.

In the literature Manin matrices have been also called \emph{right quantum} matrices \cite{K1,K2,KP,MR} or \emph{row-pseudo-commutative} matrices~\cite{CS}. For a review of their properties, and further references, see~\cite{CFRi}.

\begin{Definition}\label{def: cdet}
The \emph{column$($-ordered$)$ determinant} of an $N\times N$ matrix $M$ is
\begin{gather*} \cdet M \coloneqq \sum_{\sigma\in S_N} (-1)^{|\sigma|} M_{\sigma(1) 1} M_{\sigma(2) 2} \cdots M_{\sigma(N) N}. \end{gather*}
\end{Definition}
\begin{Lemma}\label{lem: col perms}
The column determinant $\cdet M$ changes only by a sign under the exchange of any two rows of~$M$. If $M$ is Manin, then $\cdet M$ also changes only by a sign under the exchange of any two columns of~$M$.
\end{Lemma}
\begin{proof} The first part is manifest. See \cite[Section~3.4]{CFRi} for the second.\end{proof}

\begin{Proposition}\label{prop: X block} Let $M$ be an $N\times N$ Manin matrix with coefficients in $\mc A$. Let $X$ be a~$k \times(N-k)$ matrix with coefficients in~$\mc A$, for some $0\leq k\leq N$. Then
\begin{gather*} \cdet M = \cdet \left(M \bmx 1 & X \\ 0 & 1 \emx\right). \end{gather*}
\end{Proposition}
\begin{proof}
See \cite[Section~5.1]{CFRi}.\end{proof}

This has the following corollary which will be important for us.
\begin{Proposition}\label{prop: schur complements} Let $M = \bmx A & B \\ C & D \emx$ be the block form of an $N\times N$ Manin matrix with coefficients in $\mc A$.
\begin{enumerate}[$(i)$]\itemsep=0pt
\item Suppose $\mc A$ is a subalgebra of a $($possibly larger$)$ algebra $\mc A'$ over which $A$ has a right inverse, i.e., $AA^{-1} = 1$ for some matrix $A^{-1}$ with coefficients in $\mc A'$. Then
\begin{gather*} \cdet M = \cdet A \cdet \big(D- C A^{-1} B\big) \end{gather*}
as an equality in $\mc A$.
\item Suppose $\mc A$ is a subalgebra of a $($possibly larger$)$ algebra $\mc A''$ over which $D$ has a right inverse, i.e., $DD^{-1} = 1$ for some matrix $D^{-1}$ with coefficients in $\mc A''$. Then
\begin{gather*} \cdet M = \cdet D \cdet\big(A- B D^{-1} C\big) \end{gather*}
as an equality in $\mc A$.
\end{enumerate}
\end{Proposition}
\begin{proof} We work initially over $\mc A'$. Suppose $A$ has a right inverse. By Proposition~\ref{prop: X block} we have
\begin{align*} \cdet\bmx A & B \\ C & D \emx
&= \cdet\left(\bmx A & B \\ C & D \emx \bmx 1 & -A^{-1} B \\ 0 & 1 \emx \right) \\
&= \cdet \bmx A & 0 \\ C & D- CA^{-1} B \emx
= \cdet A \cdet \big(D- C A^{-1} B\big) \end{align*}
as an equality in $\mc A'$. But $\cdet M$ belongs to $\mc A$, so in fact this is an equality in~$\mc A$. This establishes part~$(i)$.

For part $(ii)$ note that, by Lemma~\ref{lem: col perms}, $\cdet M$ is invariant under the exchange of any pair of rows followed by the exchange of the corresponding pair of columns. So we can rearrange the blocks to find
\begin{gather*} \cdet M = \cdet \bmx D & C \\ B & A \emx\end{gather*}
and then argue as for part $(i)$.
\end{proof}

\begin{Remark} The proposition above is the first half of \cite[Proposition~10]{CFRi}, specifically li\-nes~(5.17) and~(5.18). The subsequent lines (5.19) and (5.20) appear to contain misprints. For example, if $M = \bmx a & b \\ c & d \emx$ is a $2\times 2$ Manin matrix with $d$ invertible then $\cdet M = ad - cb = \big(a- c b d^{-1} \big) d = \big(a- cd^{-1} b\big) d$ whereas \cite[line~(5.20)]{CFRi} gives $\cdet M = \big(a - b d^{-1} c\big) d$, which is not in general the same.
\end{Remark}

\subsection{Quantum Gaudin model} \label{sec: quantum Gaudin bispec}

The algebra of observables of the quantum Gaudin model associated with $\gl^\D_M$ is the enveloping algebra $U\big(\gl^\D_M\big)$, equipped with its usual associative product. Let $\del_z := \frac{\del}{\del z}$ and consider the same Lax matrix given by~\eqref{formal Lax glM}, as in the classical model we considered above but now regarded as taking values in $\gl^\D_M\into U\big(\gl^\D_M\big)$. Its transpose is
\begin{gather*}
\trp \mc L^\D(z) dz = \sum_{a, b=1}^M E_{ab} \otimes \left( \mathsf E^{(\infty)}_{ab \ul 1} + \sum_{i=1}^n \sum_{r=0}^{\tau_i - 1} \frac{\mathsf E^{(z_i)}_{ab \ul r}}{(z - z_i)^{r+1}} \right) dz.
\end{gather*}

Recall the definition of the column-ordered determinant, Definition~\ref{def: cdet}, and consider the quantity
\begin{gather} \label{formal Lax glM quant}
\prod_{i=1}^n (z - z_i)^{\tau_i} \cdet\big(\del_z \mathbf 1_{M\times M} - \trp \mc L^\D(z)\big) =: \sum_{k=0}^M S_k(z) \del_z^k.
\end{gather}
This is a differential operator in $z$ of order~$M$. For each $0\leq k\leq M$, the coefficient $S_k(z)$ of $\del_z^k$ is a rational function in $z$ valued in $U\big(\gl^\D_M\big)$.

The \emph{quantum Gaudin algebra $\mathscr Z\big(\gl^\D_M\big)$} of the $\gl^\D_M$-Gaudin model is by definition the unital subalgebra of $U\big(\gl^\D_M\big)$ generated by the coefficients in the partial fraction decomposition of these rational functions~$S_k(z)$. It is a commutative subalgebra of~$U\big(\gl^\D_M\big)$,~\cite{MTV1, Talalaev}.\footnote{It is shown in~\cite{MTV1} that $\cdet\big(\del_z \mathbf 1_{M\times M} - E_{ab} \ox \sum\limits_{n=0}^\infty (\mathsf E_{ab}\ox t^n) z^{-n-1}\big)$ generates a commutative subalgebra of~$U(\gl_M[t])$. The algebra $\mathscr Z\big(\gl^\D_M\big)$ is a homomorphic image of this algebra in~$U\big(\gl^\D_M\big)$.}

The quantum Gaudin algebra $\mathscr Z\big(\gl^{\tilde \D}_N\big)$ of the $\gl^{\tilde \D}_N$-Gaudin model is defined in exactly the same way in terms of the $N^{\rm th}$ order differential operator in~$\lambda$,
\begin{gather*} \prod_{a=1}^m (\lambda - \lambda_a)^{\tilde \tau_a} \cdet \big(\del_\lambda \mathbf 1_{N\times N} - \trp \mc L^{\tilde \D}(\lambda)\big),\end{gather*}
where, cf.~\eqref{formal Lax glN},
\begin{gather*}
\trp \mc L^{\tilde \D}(\lambda) d\lambda = \sum_{i, j=1}^N \tilde{E}_{ij} \otimes \left( \tilde{\mathsf E}^{(\infty)}_{ij \ul 1} + \sum_{a=1}^m \sum_{s=0}^{\tilde \tau_a - 1} \frac{\tilde{\mathsf E}^{(\lambda_a)}_{ij \ul s}}{(\lambda - \lambda_a)^{s+1}} \right) d\lambda.
\end{gather*}
There is an automorphism of $\gl^\D_N$ defined by $\mc L^{\tilde \D}(\lambda) \mapsto -\trp \mc L^{\tilde \D}(\lambda)$. The Gaudin algebra is stabilized by this automorphism. (This statement follows from applying a tensor product of evaluation homomorphisms of Takiff algebras to the statement of \cite[Proposition~8.4]{MTV1}). Therefore we may equivalently consider the $N^{\rm th}$ order differential operator
\begin{gather} \prod_{a=1}^m (\lambda - \lambda_a)^{\tilde \tau_a} \cdet \big(\del_\lambda \mathbf 1_{N\times N} + \mc L^{\tilde \D}(\lambda)\big) =: \sum_{k=0}^N \tilde S_k(\lambda)\del_\lambda^k\label{ccd}\end{gather}
and define the quantum Gaudin algebra $\mathscr Z\big(\gl^{\tilde \D}_N\big)$ to be the unital subalgebra of $U\big(\gl^{\tilde \D}_N\big)$ generated by the coefficients in the partial fraction decomposition of the rational functions $\tilde S_k(\lambda)$ in $\lambda$. It is a commutative subalgebra of $U\big(\gl^{\tilde \D}_N\big)$.

To state our result on quantum $(\gl_M, \gl_N)$-duality, it will be convenient to write \eqref{ccd} in the equivalent form
\begin{gather*} \prod_{a=1}^m (\partial_z - \lambda_a)^{\tilde \tau_a} \cdet \big(-z \mathbf 1_{N\times N} + \mc L^{\tilde \D}(\del_z)\big)= \sum_{k=0}^N \tilde S_k(\del_z) (-z)^k.
\end{gather*}
Let us explain the meaning of the expression
\begin{gather*} \cdet \big(-z \mathbf 1_{N\times N} + \mc L^{\tilde \D}(\del_z)\big).\end{gather*} The quantity
\begin{gather*} \cdet \big(\del_\lambda \mathbf 1_{N\times N} + \mc L^{\tilde \D}(\lambda)\big),\end{gather*} which appears in~\eqref{ccd}, belongs to the algebra $U\big(\gl^{\tilde \D}_N\big)(\lambda)[\del_\lambda]$ of differential operators in~$\lambda$ whose coefficients are rational functions of~$\lambda$ with coefficients in~$U\big(\gl^{\tilde \D}_N\big)$. Here $\lambda$ and $\del_\lambda$ can be regarded as formal generators obeying the commutation relation $[\del_\lambda,\lambda]=1$. We can relabel these generators as we wish, provided we preserve this relation. In particular, we may send $(\del_\lambda,\lambda) \mapsto (-z,\del_z)$, since $[-z,\del_z]=1$. Thus $\cdet \big(-z \mathbf 1_{N\times N} + \mc L^{\tilde \D}(\del_z)\big)$ is an element of the algebra
$U\big(\gl^{\tilde \D}_N\big)(\del_z)[z]$.

More precisely, we shall be concerned in what follows with the quantity
\begin{gather} \label{formal Lax glN quant}
\prod_{a=1}^m (\del_z - \lambda_a)^{\tilde\tau_a} \cdet \big(z \mathbf 1_{N\times N} - \mc L^{\tilde \D}(\del_z)\big) = \sum_{k=0}^N (-1)^{N-k} \tilde S_k(\partial_z) z^k.
\end{gather}

\subsection{Bosonic realisation}
We consider realisations of $U\big(\gl^\D_M\big)$ and $U\big(\gl^{\tilde \D}_N\big)$ acting by differential operators on the polynomial algebra $\CC[x^a_i]_{i=1\; a=1}^{N\;\;\; M}$.
Namely, let $\del^a_i := \frac{\del}{\del x^a_i}$
and let us denote by $\mc U_{\mathsf b}$ the unital associative algebra generated by $\{x^a_i\}_{i=1\; a=1}^{N\;\;\; M}$ and $\{\del_i^a\}_{i=1\; a=1}^{N\;\;\; M}$ subject to the commutation relations
\begin{gather*}
[ x^a_i, x^b_j ] = 0 ,\quad
[ \del^a_i, x^b_j ] = \delta_{ij} \delta_{ab},\quad
[ \del^a_i, \del^b_j ] = 0,
\end{gather*}
for $a, b = 1, \ldots, M$ and $i, j = 1, \ldots, N$.

$\mc U_{\mathsf b}$ is in particular a Lie algebra, with the Lie bracket given by the commutator.

\begin{Lemma} \label{lem: pi tilde pi quant b}
The linear maps $\piq_{\mathsf b}\colon \mathfrak{gl}_M^\D \to \U_{\mathsf b}$ and $\piqt_{\mathsf b} \colon \mathfrak{gl}_N^{\tilde\D} \to \U_{\mathsf b}$ defined by
\begin{gather*}
\piq_{\mathsf b}\big(\mathsf E^{(z_i)}_{ab \ul r}\big) = \sum_{u= \nu_i+1}^{\nu_i +\tau_i - r} x^a_{u+r} \del^b_u, \qquad
\piq_{\mathsf b}\big(\mathsf E^{(\infty)}_{ab \ul 1}\big) = - \left( \bigoplus_{c=1}^m J_{\tilde \tau_c}(-\lambda_c) \right)_{ba},
\end{gather*}
for every $r = 0, \ldots, \tau_i -1$, $i = 1, \ldots, n$ and $a, b = 1,\ldots, M$, and
\begin{gather*}
\piqt_{\mathsf b}\big(\tilde{\mathsf E}^{(\lambda_a)}_{ij \ul s} \big) = \sum_{u= \tilde\nu_a+1}^{\tilde\nu_a + \tilde\tau_a - s}
\del^u_j x^{u+s}_i, \qquad
\piqt_{\mathsf b}\big(\tilde{\mathsf E}^{(\infty)}_{ij \ul 1} \big) = - \left( \bigoplus_{k=1}^n J_{\tau_k}(-z_k) \right)_{ji},
\end{gather*}
for every $s = 0, \ldots, \tilde \tau_a -1$, $i,j = 1, \ldots, N$ and $a = 1,\ldots, m$, are homomorphisms of Lie algebras. They extend uniquely to homomorphisms of associative algebras $\piq_{\mathsf b}\colon U\big(\mathfrak{gl}_M^\D\big) \to \U_{\mathsf b}$ and $\piqt_{\mathsf b}\colon U\big(\mathfrak{gl}_N^{\tilde\D}\big) \to \U_{\mathsf b}$.
\end{Lemma}

\begin{proof}For each $i,j = 1, \ldots, n$ and $a, b = 1,\ldots, M$ we have
\begin{align*}
\big[ \piq_{\mathsf b}\big(\mathsf E^{(z_i)}_{ab \ul r}\big), \piq_{\mathsf b}\big(\mathsf E^{(z_j)}_{cd \ul s}\big) \big]
& = \sum_{u= \nu_i+1}^{\nu_i +\tau_i - r} \sum_{v= \nu_j+1}^{\nu_j +\tau_j - s} \big[ x^a_{u+r} \del^b_u, x^c_{v+s} \del^d_v \big]\\
& = \sum_{u = \nu_i+1}^{\nu_i+\tau_i - r} \sum_{v = \nu_i+1}^{\nu_i+\tau_i - s} \big( x^a_{u+r} \big[ \del^b_u, x^c_{v+s} \big] \del^d_v + x^c_{v+s} \big[ x^a_{u+r}, \del^d_v \big] \del^b_u \big) \delta_{ij}\\
& = \sum_{u = \nu_i+1}^{\nu_i+\tau_i - r - s} \big( \delta_{bc} x^a_{u+r+s} \del^d_u - \delta_{ad} x^c_{u+r+s} \del^b_u \big) \delta_{ij}\\
& = \big( \delta_{bc} \piq_{\mathsf b}\big(\mathsf E^{(z_i)}_{ad \ul {r+s}}\big) - \delta_{ad} \piq_{\mathsf b}\big(\mathsf E^{(z_i)}_{cb \ul {r+s}}\big) \big) \delta_{ij} = \piq_{\mathsf b} \big( \big[ \mathsf E^{(z_i)}_{ab \ul r}, \mathsf E^{(z_j)}_{cd \ul s} \big] \big).
\end{align*}
In the second equality we have used the fact that if $i \neq j$ then all commutators vanish due to the restriction in the range of values in the sums over~$u$ and~$v$.

Likewise, for all $i,j = 1, \ldots, N$ and $a, b = 1,\ldots, m$ we find
\begin{align*}
\big[ \piqt_{\mathsf b}\big(\tilde{\mathsf E}^{(\lambda_a)}_{ij \ul r}\big), \piqt_{\mathsf b}\big(\tilde{\mathsf E}^{(\lambda_b)}_{kl \ul s}\big) \big]&=\sum_{u= \tilde\nu_a+1}^{\tilde\nu_a+ \tilde\tau_a - r}
\sum_{v= \tilde\nu_b+1}^{\tilde\nu_b+ \tilde\tau_b - s} \big[ \del_j^u x_i^{u+r}, \del_l^v x_k^{v+s} \big]\\
& =
\sum_{u= \tilde\nu_a+1}^{\tilde\nu_a+ \tilde\tau_a - r}
\sum_{v= \tilde\nu_a+1}^{\tilde\nu_a+ \tilde\tau_a - s} \big( \del_j^u \big[ x_i^{u+r}, \del_l^v \big] x_k^{v+s} + \del_l^v \big[ \del_j^u, x_k^{v+s} \big] x_i^{u+r} \big) \delta_{ab}\\
& = \sum_{u= \tilde\nu_a+1}^{\tilde\nu_a+ \tilde\tau_a - r - s} \big( {-} \delta_{il} \del_j^u x_k^{u+r+s} + \delta_{jk} \del_l^u x_i^{u+r+s} \big) \delta_{ab}\\
& = \big( \delta_{jk} \piqt_{\mathsf b}\big( \tilde{\mathsf E}^{(\lambda_a)}_{il \ul{r+s}} \big) - \delta_{il} \piqt_{\mathsf b}\big( \tilde{\mathsf E}^{(\lambda_a)}_{kj \ul{r+s}} \big) \big) \delta_{ab} = \piqt_{\mathsf b} \big( \big[ \tilde{\mathsf E}^{(\lambda_a)}_{ij \ul r}, \tilde{\mathsf E}^{(\lambda_b)}_{kl \ul s} \big] \big),
\end{align*}
as required. Moreover, all the commutators involving the generators at infinity are also easily seen to be preserved by the linear maps $\piq_{\mathsf b}$ and $\piqt_{\mathsf b}$ since $z_i \in \CC$ and $\lambda_a \in \CC$ are central in~$\U_{\mathsf b}$.
\end{proof}

Given any unital associative algebra $\mathcal U$ we denote by $\mathcal U[z,\partial_z]$ the tensor product of unital associative algebras $\mathcal U \otimes \CC[z, \partial_z]$. As in the classical setting of Section~\ref{sec: bos real cla}, consider also the unital associative algebras $\mathcal U(z)[\partial_z] \coloneqq \mathcal U \otimes \CC(z)[\partial_z]$ and $\mathcal U(\partial_z)[z] \coloneqq \mathcal U \otimes \CC(\partial_z)[z]$, both containing $\mathcal U[z, \partial_z]$ as a subalgebra. We extend the homomorphisms $\piq_{\mathsf b}$ and $\piqt_{\mathsf b}$ from Lemma \ref{lem: pi tilde pi quant b} to homomorphisms of tensor product algebras,
\begin{gather*}
\piq_{\mathsf b}\colon \ U\big(\mathfrak{gl}_M^\D\big)(z)[\partial_z] \to \mathcal U_{\mathsf b}(z)[\partial_z], \qquad
\piqt_{\mathsf b}\colon \ U\big(\mathfrak{gl}_M^{\tilde \D}\big)(\partial_z)[z] \to \mathcal U_{\mathsf b}(\partial_z)[z],
\end{gather*}
respectively. Applying these homomorphisms respectively to the expressions given by \eqref{formal Lax glM quant} and~\eqref{formal Lax glN quant}, Theorem~\ref{thm: q bispec bos} below shows that the resulting expressions in fact live in the common subalgebra $\mathcal U_{\mathsf b}[z, \partial_z]$. The coefficients of the resulting differential operators in $z$ span the respective images of the quantum Gaudin algebras in $\U_{\mathsf b}$, namely
\begin{gather*} \piq_{\mathsf b}\big(\mathscr Z\big(\gl^\D_M\big)\big) \subset \U_\mathsf{b}\qquad\text{and}\qquad
\piqt_{\mathsf b}\big(\mathscr Z\big(\gl^{\tilde \D}_N\big)\big) \subset \U_\mathsf{b}.\end{gather*}
The following theorem establishes that these commutative subalgebras of $\U_{\mathsf b}$ coincide.

\begin{Theorem} \label{thm: q bispec bos} We have
\begin{gather*}
\piq_{\mathsf b} \left( \prod_{i=1}^n (z - z_i)^{\tau_i} \cdet\big( \del_z {\bf 1}_{M \times M} - \trp \mc L^\D(z) \big) \right)\\
 \qquad{} = \piqt_{\mathsf b} \left( \prod_{a=1}^m (\del_z - \lambda_a)^{\tilde\tau_a} \cdet\big( z {\bf 1}_{N \times N} - \mc L^{\tilde \D}(\partial_z) \big) \right),
\end{gather*}
as an equality of polynomial differential operators in~$z$.
\end{Theorem}

\begin{proof}
Introduce the $M \times M$ and $N \times N$ block diagonal matrices
\begin{gather*}\Lambda \coloneqq \bigoplus_{a = 1}^m \trp J_{\tilde\tau_a}(\del_z - \lambda_a),\qquad
Z \coloneqq \bigoplus_{i = 1}^n J_{\tau_i}(z - z_i).
\end{gather*}
Also introduce the $M \times N$ matrices
\begin{gather*}
D \coloneqq (\del^a_i)_{a = 1\;i=1}^{M \;\;\;\, N},
\qquad
X \coloneqq (x^a_i)_{a = 1\;i=1}^{M \;\;\;\, N}.
\end{gather*}
Consider the block matrix
\begin{gather*}
M \coloneqq \left(
\begin{matrix}
 \Lambda & X\\
\trp D & Z
\end{matrix}
\right),
\end{gather*}
with entries in the noncommutative algebra $\mc A \coloneqq \mc U_{\mathsf b} [z, \partial_z]$.
The key observation is that this is a Manin matrix.
Indeed, the only non-trivial check is for the $2\times 2$ submatrices of the form
\begin{gather*} \bmx \del_z - \lambda_a & x_i^a \\ \del^a_i & z- z_i \emx \end{gather*}
and for these we have $[\del_z - \lambda_a, z-z_i] = 1 = [\del_i^a,x^a_i]$ as required. This fact means that we can follow the proof of Theorem \ref{thm: cla bispec bos}, with suitable modifications, as follows.

The square matrices $Z$ and $\Lambda$ with entries in $\CC[z, \partial_z] \subset \mathcal U_{\mathsf b}[z, \partial_z]$ have (two-sided) inverses in the enlarged algebras $\mc A'' \coloneqq \mc U_{\mathsf b} (z)[\partial_z]$ and $\mc A' \coloneqq \mc U_{\mathsf b}(\partial_z)[z]$, respectively, both of which contain~$\mc A$ as a subalgebra. These inverses are given explicitly by
\begin{gather*}
Z^{-1} = \bigoplus_{i = 1}^n J_{\tau_i}(z - z_i)^{-1}, \qquad
\Lambda^{-1} = \bigoplus_{a = 1}^m \trp J_{\tilde\tau_a}(\del_z - \lambda_a)^{-1}.
\end{gather*}
We are therefore in the setup of Proposition \ref{prop: schur complements}. We may apply it to evaluate $\cdet M$ in two different ways. We obtain
\begin{gather} \label{spec dual quant b}
\cdet \Lambda \cdet\big( Z - \trp D \Lambda^{-1} X\big) = \cdet Z \cdet\big( \Lambda - X Z^{-1} \trp D\big),
\end{gather}
as an equality in $\mc A = \mc U_{\mathsf b}[z, \partial_z]$, namely this is an equality of polynomial differential operators in~$z$ with coefficients in~$\mc U_{\mathsf b}$.

It remains to evaluate both sides of \eqref{spec dual quant b} more explicitly. We have
\begin{gather*}
\cdet Z = \prod_{i=1}^n (z - z_i)^{\tau_i}, \qquad \cdet \Lambda = \prod_{a=1}^m (\partial_z - \lambda_a)^{\tilde\tau_a},
\end{gather*}
where the order of the products on the right of these equalities does not matter. Now $Z$ and $\Lambda$ can be written explicitly as follows
\begin{gather*}
Z = \sum_{i,j=1}^N \tilde E_{ij} \big( z \delta_{ij} - \piqt_{\mathsf b}\big(\tilde{\mathsf E}^{(\infty)}_{ji \ul 1}\big) \big), \qquad
\Lambda = \sum_{a,b=1}^M E_{ab} \big( \partial_z \delta_{ab} - \piq_{\mathsf b}\big(\mathsf E^{(\infty)}_{ab \ul 1}\big) \big)
\end{gather*}
with $\piq_{\mathsf b}$ and $\piqt_{\mathsf b}$ as defined in Lemma~\ref{lem: pi tilde pi quant b}. In terms of these expressions we can write
\begin{align*}
\Lambda - X Z^{-1} \trp D &= \sum_{a,b = 1}^M E_{ab} \big(\Lambda - X Z^{-1} \trp D\big)_{ab}\\
& = \partial_z {\bf 1} - \sum_{a, b= 1}^M E_{ab} \left( \piq_{\mathsf b}\big(\mathsf E^{(\infty)}_{ab \ul 1}\big) + \sum_{i=1}^n
\sum_{j, k = \nu_i+1}^{\nu_i+ \tau_i} x^a_j \big( J_{\tau_i}(z - z_i)^{-1} \big)_{jk} \partial^b_k \right).
\end{align*}
The latter expression is exactly $\partial_z {\bf 1} - \piq_{\mathsf b}\big( \trp \mc L^\D(z) \big)$ by virtue of Lemma~\ref{lem: pi tilde pi quant b}, the expression~\eqref{formal Lax glM} for the Lax matrix $\mc L^\D(z)$ and the expression~\eqref{J block inv} for the inverse of a Jordan block. Likewise
\begin{align*}
Z - \trp D \Lambda^{-1} X &= \sum_{i, j = 1}^N \tilde E_{ij} \big(Z - \trp D \Lambda^{-1} X\big)_{ij}\\
& = z {\bf 1} - \sum_{i,j = 1}^N \tilde E_{ij} \left( \piqt_{\mathsf b}\big( \tilde {\mathsf E}^{(\infty)}_{ji \ul 1} \big) + \sum_{a = 1}^m
\sum_{b,c = \tilde\nu_a+1}^{\tilde\nu_a+ \tilde\tau_a}
\partial^b_i \big( J_{\tilde\tau_a}(\partial_z - \lambda_a)^{-1} \big)_{cb} x^c_j \right),
\end{align*}
which coincides with $z {\bf 1} - \piqt_{\mathsf b} \big( \mc L^{\tilde \D}(\partial_z) \big)$. The result now follows.
\end{proof}

In the special case of no Jordan blocks and no non-trivial Takiff algebras, Theorem~\ref{thm: q bispec bos} can be found in~\cite{MTVcapelli}. See also \cite[Proposition~8]{CF}, where it is noted that the relation $\cdet M= \det Z \cdet\big( \Lambda - X Z^{-1} \trp D \big)$ leads to a relation between the classical spectral curve and the ``quantum spectral curve''.

\section[$\ZZ_2$-cyclotomic Gaudin models with irregular singularities]{$\boldsymbol{\ZZ_2}$-cyclotomic Gaudin models with irregular singularities} \label{sec: cyclo bispec}

Another possible class of generalisations of Gaudin models are those whose Lax matrix is equivariant under an action of the cyclic group, determined by a choice of automorphism of the Lie algebra (here~$\gl_M$).
Such models were considered in \cite{Skr3,Skr2,Skr1} and in~\cite{CY} for automorphisms of order~2, and for automorphisms of arbitrary finite order in~\cite{VY1,VY3}.

It is natural to ask whether $(\gl_M, \gl_N)$-dualities also exist, in the sense of Section~\ref{sec: c}, between cyclotomic Gaudin models. Theorem~\ref{thm: cla bispec bos cycl}, which can be deduced from the results of~\cite{Adams:1990mj}, establishes a~duality between a cyclotomic $\gl_M$-Gaudin model associated with the diagram automorphism of~$\gl_M$ and a non-cyclotomic $\sp_N$-Gaudin model.

\subsection[$\ZZ_2$-cyclotomic Lax matrix for the diagram automorphism]{$\boldsymbol{\ZZ_2}$-cyclotomic Lax matrix for the diagram automorphism} \label{sec: Z2 cyclo}

Let $z_i \in \CC$ for $i =1, \ldots, n$ be such that $0\neq z_i \neq \pm z_j$ for $i \neq j$. Pick and fix integers $\tau_i \in \ZZ_{\geq 1}$ for $i=0$ and for each $i = 1, \ldots, n$. Consider the effective divisor
\begin{gather*}
\C = 2 \tau_0\cdot 0 + \sum_{i=1}^n \tau_i \cdot z_i + \sum_{i=1}^n \tau_i \cdot (-z_i) + 2 \cdot \infty.
\end{gather*}
Note, in particular, that the Takiff degree at the origin is always even. Let $N\in \ZZ_{\geq 1}$. We require that $\deg \C = 2N + 2$ or in other words,
\begin{gather*}
\tau_0 + \sum_{i=1}^n \tau_i = N.
\end{gather*}

Let $M\in \ZZ_{\geq 1}$. As before, cf.\ Section~\ref{sec: glND}, denote by $\mathsf E_{ab}$ for $a, b = 1, \ldots, M$ the standard basis of $\mathfrak{gl}_M$. There is an automorphism $\sigma$ of $\gl_M$ defined by
\begin{gather*} \sigma(\mathsf E_{ab}) := - \mathsf E_{ba}.\end{gather*}
We call this the \emph{diagram automorphism} of $\gl_M$. The Lie algebra $\gl_M$ decomposes into the direct sum of the $\pm 1$ eigenspaces of~$\sigma$,
\begin{gather*} \gl_M = \so_M \oplus \mf p_M.\end{gather*}
Here the subalgebra of invariants, i.e., the $(+1)$-eigenspace, is a copy of the Lie algebra $\so_M$. The $(-1)$-eigenspace $\p_M$ is a copy of the symmetric second rank tensor representation of $\so_M$. We shall write
\begin{gather*}
\mathsf E_{ab}^\pm \coloneqq \mathsf E_{ab} \pm \mathsf E_{ba},
\end{gather*}
so that $\mathsf E^+_{ab} \in \so_M$ and $\mathsf E^-_{ab} \in \p_M$, for all $a, b = 1, \ldots, M$. We introduce the pair of maps $\Pi_{(0)} \colon \gl_M \to \so_M$, $\mathsf E_{ab} \mapsto \mathsf E^-_{ab}$ and $\Pi_{(1)} \colon \gl_M \to \p_M$, $\mathsf E_{ab} \mapsto \mathsf E^+_{ab}$. More generally, for $r \in \ZZ_{\geq 0}$ we define $\Pi_{(r)} \coloneqq \Pi_{(r \, \text{mod}\, 2)}\colon \gl_M \to \gl_M$, so that $\Pi_{(r)} \mathsf E_{ab} = \mathsf E_{ab} - (-1)^r \mathsf E_{ba}$.

There is an extension of the automorphism $\sigma$ to an automorphism of the polynomial algebra $\gl_M[\varepsilon]$ defined by
\begin{gather*} \mathsf X \varepsilon^k \mapsto \sigma(\mathsf X) (-\varepsilon)^k.\end{gather*}
Let $\gl_M[\varepsilon]^\sigma$ denote the subalgebra of invariants. As vector spaces, we have
\begin{gather*} \gl_M[\varepsilon]^\sigma \cong \so_M \big[\varepsilon^2\big] \oplus \varepsilon \p_M\big[\varepsilon^2\big].
\end{gather*}
Define $\mf{gl}_M^{\C}$ to be the direct sum of Takiff Lie algebras
\begin{align*}
\mathfrak{gl}_M^\C &\coloneqq (\varepsilon_\infty \gl_M[\varepsilon_\infty])^\sigma / \varepsilon_\infty^2 \oplus \bigoplus_{i=1}^n \mathfrak{gl}_M[\varepsilon_{z_i}] / \varepsilon_{z_i}^{\tau_i} \oplus \gl_M[\varepsilon_0]^\sigma / \varepsilon_0^{2 \tau_0} .
\end{align*}
Note that as a vector space the Takiff algebra attached to the point at infinity is simply $(\varepsilon_\infty \gl_M[\varepsilon_\infty])^\sigma / \varepsilon_\infty^2 \cong \p_M \varepsilon_\infty$.

As before we let $\rho \colon \mathfrak{gl}_M \to \operatorname{Mat}_{M \times M}(\CC)$ denote the defining representation of $\mathfrak{gl}_M$ and write $E_{ab} \coloneqq \rho(\mathsf E_{ab})$. The formal Lax matrix of the $\ZZ_2$-cyclotomic Gaudin model associated with $\mathfrak{gl}^{\mathcal C}_M$ is the $M\times M$ matrix with entries consisting of $\gl_M^\C$-valued rational functions of $z$, given by
\begin{align}
\tilde{\mc L}^\C(z) dz
&\coloneqq \sum_{a, b=1}^M E_{ba} \otimes \left( \mathsf E^{+(\infty)}_{ab \ul 1} + \sum_{r=0}^{2 \tau_0 - 1} \frac{(\Pi_{(r)} \mathsf E_{ab})^{(0)}_{\ul r}}{z^{r+1}} \right.\notag\\
&\left.\qquad\qquad\qquad {}+ \sum_{i=1}^n \sum_{r=0}^{\tau_i - 1} \frac{\mathsf E^{(z_i)}_{ab \ul r}}{(z - z_i)^{r+1}} + \sum_{i=1}^n \sum_{r=0}^{\tau_i - 1} \frac{(-1)^{r+1} \mathsf E^{(z_i)}_{ba \ul r}}{(z + z_i)^{r+1}} \right) dz.\label{formal Lax glM cyclotomic}
\end{align}
It obeys the following Lax algebra
\begin{gather} \label{PB Lax cyclo}
\big[ \tilde{\mc L}^\C_\1(z), \tilde{\mc L}^\C_\2(w) \big] = \big[ r_{\1\2}(z,w), \tilde{\mc L}^\C_\1(z) \big]
 - \big[ r_{\2\1}(w,z), \tilde{\mc L}^\C_\2(w) \big]
\end{gather}
where $r_{\1\2}(z,w)$ denotes the (non-skew-symmetric) classical $r$-matrix
\begin{gather*}
r_{\1\2}(z,w) \coloneqq \sum_{a,b=1}^M \left( \frac{E_{ba} \otimes E_{ab}}{w-z} - \frac{E_{ba} \otimes E_{ba}}{w+z} \right).
\end{gather*}

Consider the quantity
\begin{gather*} 
\left(z^{2 \tau_0} \prod_{i=1}^n (z - z_i)^{\tau_i}(z + z_i)^{\tau_i}\right) \det\big(\lambda \mathbf 1_{M\times M} - \tilde{\mathcal L}^\C(z)\big)
\end{gather*}
This is a polyomial in $\lambda$ of order $M$. For each $0\leq k\leq M$, the coefficient of $\lambda^k$ is a rational function in $z$ valued in $S\big(\gl^\C_M\big)$. The \emph{classical cyclotomic Gaudin algebra $\mathscr Z\big(\gl^\C_M\big)$} associated with the divisor $\C$ and the diagram automorphism $\sigma$ is by definition the Poisson subalgebra of $S\big(\gl^\C_M\big)$ generated by the coefficients of these rational functions. It follows from~\eqref{PB Lax cyclo} that~$\mathscr Z\big(\gl^\C_M\big)$ is a~Poisson-commutative subalgebra of~$S\big(\gl^\C_M\big)$.

\subsection[Lax matrix of $\sp_{2N}$-Gaudin model with regular singularities]{Lax matrix of $\boldsymbol{\sp_{2N}}$-Gaudin model with regular singularities} \label{sec: so2N}

Denote by $\tilde{\mathsf E}_{IJ}$ the standard basis of $\gl_{2N}$, where, for convenience, we shall let $I$, $J$ run over the index set $\mathcal I \coloneqq \{-N,\dots,-1,1, \dots, N\}$. There is a subalgebra of $\gl_{2N}$, isomorphic to the Lie algebra $\sp_{2N}$, spanned by
\begin{gather*} 
\bar{\mathsf E}_{IJ} \coloneqq \tilde{\mathsf E}_{IJ} - \sigma_I \sigma_J \tilde{\mathsf E}_{-J,-I},
\end{gather*}
for all $I, J \in \mathcal I$. Here we denote by $\sigma_I$ the sign of $I$, equal to~$1$ if $I > 0$ and to $-1$ if $I < 0$. We have the relation $\bar{\mathsf E}_{-J, -I} = - \sigma_I \sigma_J \bar{\mathsf E}_{IJ}$ for every $I, J \in \mathcal I$. Let
\begin{gather*}
\mathcal I_2 \coloneqq \big\{ (I, J) \in \mathcal I \times \mathcal I \,\big|\, I, J >0 \; \text{or}\; \sigma_I \sigma_J = -1 \; \text{with}\; |I| \leq |J| \big\}.
\end{gather*}
Then $\big\{ \bar{\mathsf E}_{IJ} \big\}_{(I, J) \in \mathcal I_2}$ is a basis of the subalgebra~$\sp_{2N}$. A~dual basis with respect to half the trace in the fundamental representation is given by $\big\{ \bar{\mathsf E}^{IJ} \big\}_{(I, J) \in \mathcal I_2}$ where
\begin{gather*} 
\bar{\mathsf E}^{IJ} \coloneqq \tilde{\mathsf E}_{JI} - \sigma_I \sigma_J \tilde{\mathsf E}_{-I,-J}, \qquad
\bar{\mathsf E}^{I, -I} \coloneqq \tilde{\mathsf E}_{-I, I},
\end{gather*}
for any $I, J \in \mathcal I$ with $J \neq - I$. Indeed, if we let $\bar E_{IJ} \coloneqq \rho\big(\bar{\mathsf E}_{IJ}\big)$ and $\bar E^{IJ} \coloneqq \rho\big(\bar{\mathsf E}^{IJ}\big)$ for all $I, J \in \mathcal I$ then we have $\ha \tr\big(\bar E_{IJ} \bar E^{KL}\big) = \delta_{IL} \delta_{JK}$ for all $(I, J), (K, L) \in \mathcal I_2$.

Let $\bar\D$ denote the special case of the effective divisor $\tilde\D$ of Section~\ref{sec: glND} obtained by setting $\tilde\tau_a=1$ for each $a = 1,\ldots, m$, and hence $m=M$. That is,
\begin{gather} \label{bar D}
\bar\D = \sum_{a=1}^M \lambda_a + 2\cdot \infty.
\end{gather}
Introduce the direct sum of Lie algebras
\begin{gather*}
\sp_{2N}^{\bar\D} \coloneqq \tilde\varepsilon_\infty \sp_{2N}[\tilde\varepsilon_\infty] / \tilde\varepsilon_\infty^2 \oplus \bigoplus_{a=1}^M \sp_{2N}.
\end{gather*}

The Lax matrix of the classical Gaudin model associated with the divisor $\bar\D$ is the $2N\times 2N$ matrix of $\sp_{2N}^{\bar\D}$-valued rational functions of $\lambda$ given by
\begin{align}\label{formal Lax so2N}
\mc L^{\bar\D}(\lambda) d\lambda &\coloneqq \sum_{(I,J) \in \mathcal I_2} \bar E^{IJ} \otimes
\left( \bar{\mathsf E}^{(\infty)}_{IJ} + \sum_{a=1}^M \frac{\bar{\mathsf E}^{(\lambda_a)}_{IJ}}{\lambda - \lambda_a} \right) d\lambda,
\end{align}
where by abuse of notation we drop the subscript on the Takiff generators, namely we define $\bar{\mathsf E}^{(\lambda_a)}_{IJ} \coloneqq \bar{\mathsf E}^{(\lambda_a)}_{IJ \ul 0}$ for all $a = 1, \ldots, M$ and $\bar{\mathsf E}^{(\infty)}_{IJ} \coloneqq \bar{\mathsf E}^{(\infty)}_{IJ \ul 1}$. It obeys the Lax algebra
\begin{gather} \label{PB Lax sp2N}
\big[ \mc L^{\bar\D}_\1(\lambda), \mc L^{\bar\D}_\2(\mu) \big] = \big[ \bar r_{\1\2}(\lambda, \mu), \mc L^{\bar\D}_\1(\lambda) + \mc L^{\bar\D}_\2(\mu) \big]
\end{gather}
where $\bar r_{\1\2}(\lambda, \mu)$ is the standard skew-symmetric classical $r$-matrix with spectral parameter for the Lie algebra $\sp_{2N}$, namely
\begin{gather*}
\bar r_{\1\2}(\lambda, \mu) \coloneqq \sum_{(I,J) \in \mathcal I_2} \frac{\bar E^{IJ} \otimes \bar E_{IJ}}{\mu - \lambda}.
\end{gather*}
Just as in Section~\ref{sec: spectral curves} we may consider the subalgebra $\mathscr Z\big(\sp_{2N}^{\bar \D}\big)$ of the Poisson algebra $S\big(\sp_{2N}^{\bar \D}\big)$ generated by the coefficients rational functions in $\lambda$ obtained as the coefficients of the polynomial in~$z$ defined by
\begin{gather*} 
\prod_{a=1}^M (\lambda - \lambda_a) \det \big(z \mathbf 1_{N\times N} - \mc L^{\bar \D}(\lambda)\big),
\end{gather*}
which is Poisson-commutative by virtue of the relation~\eqref{PB Lax sp2N}.

\subsection{Bosonic realisation}
Consider the Poisson algebra $\P_{\mathsf b} \coloneqq \CC[x^a_i, p^b_j]_{i,j=1\; a,b=1}^{N\quad\;\, M}$, as in Section~\ref{sec: bos real cla}, with Poisson brackets given by~\eqref{PB for Pb}.

We now want to break up the list of integers from $1$ to $N$ into $n+1$ blocks of size $\tau_i$ for each $i = 0,1, \ldots, n$.
Define the integers $\nu_i$ by -- in contrast to \eqref{nudef} --
\begin{gather*}
\nu_i \coloneqq \sum_{j=0}^{i-1} \tau_j,
\end{gather*}
for each $i = 0, \ldots, N$ (note in particular that now $\nu_0 = 0$), so that
\begin{gather*}
(1, \ldots, N) =
(1, \dots, \tau_0;\nu_1+1, \dots, \nu_1+\tau_1; \dots; \nu_{n}+1, \dots, \nu_{n}+\tau_{n}).
\end{gather*}

\newcommand{\picb}{\bar\pi}
\begin{Lemma} \label{lem: cyc pic pict}
Let $\mu \in \CC$ be arbitrary and define a pair of linear maps $\pic_{\mathsf b}\colon \mathfrak{gl}_M^\C \to \P_{\mathsf b}$ and $\picb_{\mathsf b} \colon \sp_{2N}^{\bar\D} \to \P_{\mathsf b}$ by
\begin{gather*}
\pic_{\mathsf b}\big(\mathsf E^{(z_i)}_{ab \ul r}\big) = \sum_{u= \nu_i+1}^{\nu_i +\tau_i - r}
x^a_{u+r} p^b_u, \qquad \pic_{\mathsf b}\big(\mathsf E^{+ (\infty)}_{ab \ul 1}\big) = \lambda_a \delta_{ab},\\
\pic_{\mathsf b}\big((\Pi_{(s)} \mathsf E_{ab})^{(0)}_{\ul s}\big) = \sum_{u=1}^{ \tau_0 - s}
 \big( x^a_{u+s} p^b_u - (-1)^s x^b_{u+s} p^a_u \big) - \mu \sum_{\substack{u,v=1 \\u+v=s+1 }}^{\tau_0} (-1)^v x^a_u x^b_v
\end{gather*}
for every $r = 0, \ldots, \tau_i - 1$, $s = 0, \ldots, 2 \tau_0 - 1$, $i = 1, \ldots, n$ and $a, b = 1,\ldots, M$, and
\begin{gather*}
\picb_{\mathsf b}\big(\bar{\mathsf E}^{(\lambda_a)}_{ij} \big) = p^a_j x^a_i \qquad
\picb_{\mathsf b}\big(\bar{\mathsf E}^{(\lambda_a)}_{i,-j} \big) = - x^a_j x^{a}_i, \qquad
\picb_{\mathsf b}\big(\bar{\mathsf E}^{(\lambda_a)}_{-i,j} \big) = p^a_j p^{a}_i, \\
\picb_{\mathsf b}\big(\bar{\mathsf E}^{(\infty)}_{IJ} \big) = - \left( \bigoplus_{i=n}^1 \big( -J_{\tau_i}(- z_i) \big) \oplus \big( - J_{\tau_0}(0) \big) \oplus J_{\tau_0}(0) \oplus \bigoplus_{i=1}^n J_{\tau_i}(-z_i) + \mu \tilde E_{1,-1} \right)_{JI},
\end{gather*}
for every $i,j = 1, \ldots, N$, $I, J \in \mathcal I$ and $a = 1,\ldots, m$. These maps are homomorphisms of Lie algebras. They extend uniquely to homomorphisms of Poisson algebras $\pic_{\mathsf b} \colon S\big(\mathfrak{gl}_M^\C\big) \to \P_{\mathsf b}$ and $\picb_{\mathsf b} \colon S\big(\sp_{2N}^{\bar\D}\big) \to \P_{\mathsf b}$.
\end{Lemma}

\begin{proof}We first show that $\pic_{\mathsf b}$ is a homomorphism. It follows, exactly as in the proof of Lemma~\ref{lem: pic pict} (see Lemma~\ref{lem: pi tilde pi quant b}) that
\begin{gather} \label{Eab Ecd Cycl result}
\big\{ \pic_{\mathsf b}\big(\mathsf E^{(z_i)}_{ab \ul r}\big), \pic_{\mathsf b}\big(\mathsf E^{(z_j)}_{cd \ul s}\big) \big\} = \pic_{\mathsf b}\big( \big[ \mathsf E^{(z_i)}_{ab \ul r}, \mathsf E^{(z_j)}_{cd \ul s} \big] \big),
\end{gather}
for any $r,s=0, \ldots, \tau_i-1$, $i,j = 1, \ldots, n$ and $a,b,c,d = 1, \ldots, M$. We also clearly have
\begin{gather*}
\big\{ \pic_{\mathsf b}\big((\Pi_{(s)} \mathsf E_{ab})^{(0)}_{\ul s}\big), \pic_{\mathsf b}\big(\mathsf E^{(z_i)}_{cd \ul r} \big) \big\} = 0
\end{gather*}
for any $r = 0, \ldots, \tau_i - 1$ $i = 1, \ldots, n$ and $a,b,c,d = 1, \ldots, M$ since the canonical variables entering each argument of the Poisson brackets mutually commute.

To simplify the notation, introduce $y^{ab}_r \coloneqq \sum\limits_{u=1}^{\tau_0 - r} \big( x^a_{u+r} p^b_u - (-1)^r x^b_{u+r} p^a_u \big)$. We can then write
\begin{gather*}
\pic_{\mathsf b}\big((\Pi_{(s)} \mathsf E_{ab})^{(0)}_{\ul s}\big) = y^{ab}_s - \mu \sum_{\substack{u,v=1 \\u+v=s+1 }}^{\tau_0} (-1)^v x^a_u x^b_v.
\end{gather*}
By a similar computation to the one leading to~\eqref{Eab Ecd Cycl result}, we find that
\begin{gather*}
\big\{ y^{ab}_r, y^{cd}_s \big\} = \delta_{bc} y^{ad}_{r+s} + (-1)^s \delta_{ac} y^{db}_{r+s} + (-1)^r \delta_{bd} y^{ca}_{r+s} + (-1)^{r+s} \delta_{ad} y^{bc}_{r+s}.
\end{gather*}
Likewise, we have
\begin{gather*}
 - \sum_{\substack{v,w=1\\ v+w = s+1}}^{\tau_0} (-1)^w \big\{ y^{ab}_r, x^c_v x^d_w \big\}\\
\qquad{} = - \underset{u+w = r+s+1}{\sum_{u=r+1}^{\tau_0} \sum_{w=1}^s}
(-1)^w \big( \delta_{bc} x^a_u x^d_w + (-1)^r \delta_{bd} x^c_u x^a_w + (-1)^s \delta_{ac} x^d_u x^b_w + (-1)^{r+s} \delta_{ad} x^b_u x^c_w \big).
\end{gather*}
and also by symmetry we obtain
\begin{gather*}
- \sum_{\substack{v,w=1\\ v+w = r+1}}^{\tau_0} (-1)^w \big\{ x^a_v x^b_w, y^{cd}_s \big\} = \sum_{\substack{v,w=1\\ v+w = r+1}}^{\tau_0} (-1)^w \big\{ y^{cd}_s, x^a_v x^b_w \big\}\\
\qquad{} = - \underset{u+w = r+s+1}{\sum_{u=1}^r \sum_{w=s+1}^{\tau_0}} (-1)^w \big( \delta_{bc} x^a_u x^d_w + (-1)^r \delta_{bd} x^c_u x^a_w + (-1)^s \delta_{ac} x^d_u x^b_w + (-1)^{r+s} \delta_{ad} x^b_u x^c_w \big).
\end{gather*}
It now follows by combining all the above that
\begin{gather*}
\big\{ \pic_{\mathsf b}\big((\Pi_{(r)} \mathsf E_{ab})^{(0)}_{\ul r}\big), \pic_{\mathsf b}\big((\Pi_{(s)} \mathsf E_{cd})^{(0)}_{\ul s}\big) \big\}\\
\qquad {}= \delta_{bc} \pic_{\mathsf b}\big((\Pi_{(r)} \mathsf E_{ad})^{(0)}_{\ul{r+s}}\big) + (-1)^s \delta_{ac} \pic_{\mathsf b}\big((\Pi_{(r)} \mathsf E_{db})^{(0)}_{\ul{r+s}}\big)\\
\qquad\quad {}+ (-1)^r \delta_{bd} \pic_{\mathsf b}\big((\Pi_{(r)} \mathsf E_{ca})^{(0)}_{\ul{r+s}}\big) + (-1)^{r+s} \delta_{ad} \pic_{\mathsf b}\big((\Pi_{(r)} \mathsf E_{bc})^{(0)}_{\ul{r+s}}\big)\\
\qquad {}= \pic_{\mathsf b}\big( \big[ (\Pi_{(r)} \mathsf E_{ab})^{(0)}_{\ul r}, (\Pi_{(s)} \mathsf E_{cd})^{(0)}_{\ul s} \big] \big),
\end{gather*}
as required. And finally, since $\mathsf E^{+ (\infty)}_{ab \ul 1}$ is a Casimir and is sent to a constant under~$\pic_{\mathsf b}$, all Poisson brackets involving it are preserved by~$\pic_{\mathsf b}$.

We now turn to showing that $\picb_{\mathsf b}$ is also a homomorphism. Define $q^a_I$ for each $I \in \mathcal I$ and $a = 1, \ldots, M$ by letting $q^a_i \coloneqq x^a_i$ and $q^a_{-i} \coloneqq p^a_i$ for every $i = 1, \dots, N$. In this notation the Poisson brackets \eqref{PB for Pb} can be rewritten more uniformly as
\begin{gather*}
\big\{ q^a_I, q^b_J \big\} = \sigma_J \delta_{I, -J} \delta_{ab},
\end{gather*}
for all $I, J \in \mathcal I$ and $a, b = 1, \ldots, M$. Moreover, we also have $\picb_{\mathsf b}\big( \bar{\mathsf E}^{(\lambda_a)}_{IJ} \big) = \sigma_J q^a_I q^a_{-J}$ for all $I, J \in \mathcal I$ and $a = 1, \ldots, M$. We then have
\begin{gather*}
 \big\{ \picb_{\mathsf b}\big( \bar{\mathsf E}^{(\lambda_a)}_{IJ} \big), \picb_{\mathsf b}\big( \bar{\mathsf E}^{(\lambda_b)}_{KL} \big) \big\}\\
 \qquad {} = \sigma_J \sigma_L \big( \sigma_K \delta_{I, -K} q^a_{-J} q^a_{-L} + \sigma_{-L} \delta_{I, L} q^a_K q^a_{-J}
 + \sigma_K \delta_{J, K} q^a_I q^a_{-L} + \sigma_{-L} \delta_{J,-L} q^a_I q^a_K \big) \delta_{ab}\\
\qquad{} = \sigma_J \sigma_K \big( \picb_{\mathsf b}\big( \bar{\mathsf E}^{(\lambda_a)}_{IL} \big) \delta_{J, K} + \picb_{\mathsf b}\big( \bar{\mathsf E}^{(\lambda_a)}_{-J, -K} \big) \delta_{I, L}
 + \picb_{\mathsf b}\big( \bar{\mathsf E}^{(\lambda_a)}_{I, -K} \big) \delta_{- J,L} + \picb_{\mathsf b}\big( \bar{\mathsf E}^{(\lambda_a)}_{-J, L} \big) \delta_{K, -I} \big) \delta_{ab}\\
 \qquad{} = \picb_{\mathsf b}\big( \big[ \bar{\mathsf E}^{(\lambda_a)}_{IJ}, \bar{\mathsf E}^{(\lambda_b)}_{KL} \big] \big),
\end{gather*}
where in the second equality we have made use of the fact that $\sigma_I \sigma_{-I} = -1$ for any $I \in \mathcal I$. Finally, the Poisson brackets involving the generators $\bar{\mathsf E}^{(\infty)}_{IJ}$ attached to infinity are all trivially preserved by~$\picb_{\mathsf b}$.
\end{proof}

We are now in a position to prove the analogue of Theorem~\ref{thm: cla bispec bos} in the present context.
\begin{Theorem} \label{thm: cla bispec bos cycl}
For any $\mu \in \CC$ as in Lemma~{\rm \ref{lem: cyc pic pict}}, we have the relation
\begin{gather*}
 \pic_{\mathsf b} \left( z^{2 \tau_0} \prod_{i=1}^n (z - z_i)^{\tau_i} (z + z_i)^{\tau_i} \det\big( \lambda {\bf 1}_{M \times M} - \tilde{\mc L}^\C(z) \big) \right)\\
\qquad {}= \picb_{\mathsf b} \left( \prod_{a=1}^M (\lambda - \lambda_a) \det\big( z {\bf 1}_{N \times N} - \mc L^{\tilde\D}(\lambda) \big) \right).
\end{gather*}
\end{Theorem}

\begin{proof}We follow the argument given in the proof of Theorem \ref{thm: cla bispec bos} very closely. Consider the $M \times M$ and $2N \times 2N$ block matrices
\begin{gather*}
\Lambda \coloneqq \big( (\lambda - \lambda_a) \delta_{ab} \big)_{a, b = 1}^M, \\
Z \coloneqq \bigoplus_{i=n}^1 \big( {}-J_{\tau_i}(-z - z_i) \big) \oplus \big( {-} J_{\tau_0}(-z) \big) \oplus J_{\tau_0}(z) \oplus \bigoplus_{i=1}^n J_{\tau_i}(z - z_i) + \mu \tilde E_{1, -1}.
\end{gather*}
We use here the convention, cf.\ Section~\ref{sec: so2N}, that indices on components of the $2N \times 2N$ matrix $Z$ run through the index set $\mathcal I = \{ -N, \ldots, -1, 1, \ldots, N \}$.
As an example of the form of the matrix $Z$, if $n=2$, $\tau_0=2$, $\tau_1=1$ and $\tau_2=2$ then we have
\begin{gather*}
Z = \begin{pmatrix}z+{{z}_{2}} & 0 & & & & & & & & \cr 1 & z+{{z}_{2}} & & & & & & & 0 & \cr & & z+{{z}_{1}} & & & & & & & \cr & & & z & 0 & 0 & 0 & & & \cr & & & 1 & z & 0 & 0 & & & \cr & & & 0 & \mu & z & 0 & & & \cr & & & 0 & 0 & -1 & z & & & \cr & & & & & & & z-{{z}_{1}} & & \cr & 0 & & & & & & & z-{{z}_{2}} & 0 \cr & & & & & & & & -1 & z-{{z}_{2}}\end{pmatrix}.
\end{gather*}

We define a pair of $M \times 2N$ matrices $P$ and $X$, whose columns are also indexed by the set~$\mathcal I$, as
\begin{gather*}
\trp P \coloneqq \bmx
-x_N^1 & \dots & -x_N^M\\
\vdots & \ddots& \vdots\\
-x_1^1 & \dots & -x_1^M\\
 p_1^1 & \dots & p_1^M\\
\vdots & \ddots& \vdots\\
 p_N^1 & \dots & p_N^M\emx,
\qquad
X \coloneqq \bmx p_N^1 & \dots & p_1^1 & x_1^1 & \dots & x_N^1 \\
\vdots & \ddots & \vdots &\vdots & \ddots & \vdots \\
p_N^M & \dots & p_1^M & x_1^M & \dots & x_N^M\emx.
\end{gather*}
Consider now the block $(M+2N) \times (M+2N)$ square matrix \eqref{clasMdef} with $\Lambda$, $Z$, $X$ and $P$ defined as above. Now the derivation leading to the equation \eqref{clas bos master eq} from the proof of Theorem \ref{thm: cla bispec bos} still holds and so it just remains to compute the determinants appearing on both sides of this identity.

On the one hand, we have
\begin{gather*}
 \Lambda - X Z^{-1} \trp P = \sum_{a,b = 1}^M E_{ab} \big(\Lambda - X Z^{-1} \trp P\big)_{ab}\\
\qquad{} = \lambda {\bf 1} - \sum_{a, b= 1}^M E_{ab} \left( \pic_{\mathsf b}\big(\mathsf E^{+ (\infty)}_{ab \ul 1}\big) + \sum_{i=1}^n \sum_{j, k = \nu_i+1}^{\nu_i+ \tau_i} x^a_j \big(Z^{-1}\big)_{jk} p^b_k + \sum_{j, k = 1}^{\tau_0} x^a_j \big(Z^{-1}\big)_{jk} p^b_k\right.\\
\left.\qquad\quad {}
- \sum_{j, k = 1}^{\tau_0} p^a_j \big(Z^{-1}\big)_{-j, -k} x^b_k - \sum_{j, k = 1}^{\tau_0} x^a_j \big(Z^{-1}\big)_{j, -k} x^b_k
- \sum_{i=1}^n \sum_{j, k = \nu_i+1}^{\nu_i+ \tau_i} p^a_j \big(Z^{-1}\big)_{-j, -k} x^b_k \right).
\end{gather*}
For each $i = 1, \ldots, n$ we note using the expression \eqref{J block inv} for the inverse of a Jordan block together with Lemma~\ref{lem: cyc pic pict} that
\begin{gather*}
\sum_{j, k = \nu_i+1}^{\nu_i+ \tau_i} x^a_j \big(Z^{-1}\big)_{jk} p^b_k = \sum_{r=0}^{\tau_i - 1} \frac{\pic_{\mathsf b}\big( \mathsf E_{ab \ul r}^{(z_i)} \big)}{(z - z_i)^{r+1}},\\
- \sum_{j, k = \nu_i+1}^{\nu_i+ \tau_i} p^a_j \big(Z^{-1}\big)_{-j, -k} x^b_k = \sum_{r=0}^{\tau_i - 1} \frac{(-1)^{r+1} \pic_{\mathsf b}\big( \mathsf E_{ba \ul r}^{(z_i)} \big)}{(z + z_i)^{r+1}}.
\end{gather*}
Next, for the two terms in the middle line above, corresponding to the origin, we find
\begin{gather*}
\sum_{j, k = 1}^{\tau_0} \big( x^a_j \big(Z^{-1}\big)_{jk} p^b_k - p^a_j \big(Z^{-1}\big)_{-j, -k} x^b_k \big) = \sum_{s=0}^{\tau_0 - 1} \frac{1}{z^{s+1}} \sum_{u=1}^{ \tau_0 - s} \big( x^a_{u+s} p^b_u - (-1)^s x^b_{u+s} p^a_u \big).
\end{gather*}
Finally, for the remaining term we have
\begin{gather*}
- \sum_{j, k = 1}^{\tau_0} x^a_j \big(Z^{-1}\big)_{j, -k} x^b_k = - \sum_{s=1}^{2 \tau_0 - 1} \frac{\mu}{z^{s+1}} \sum_{\substack{u, v = 1\\ u+v = s+1}}^{\tau_0} (-1)^v x^a_u x^b_v.
\end{gather*}
Putting all the above together we deduce that $\Lambda - X Z^{-1} \trp P = \lambda {\bf 1} - \pic_{\mathsf b} \big( \trp \tilde{\mathcal L}^{\mathcal C}(z) \big)$.

On the other hand, we have
\begin{align*}
Z - \trp P \Lambda^{-1} X &= \sum_{I,J \in \mathcal I} \tilde E_{IJ} \big(Z - \trp P \Lambda^{-1} X\big)_{IJ}\\
&= z {\bf 1} - \sum_{(I,J) \in \mathcal I_2} \bar E^{IJ} \left( \picb_{\mathsf b}\big( \bar{\mathsf E}^{(\infty)}_{IJ} \big) - \sum_{a = 1}^M \frac{\picb_{\mathsf b}\big( \bar{\mathsf E}^{(\lambda_a)}_{IJ} \big)}{\lambda - \lambda_a} \right) = z {\bf 1} - \picb_{\mathsf b}\big( \mathcal L^{\bar\D}(\lambda) \big).
\end{align*}
To see the second equality we note that setting $z=0$ in $Z - \trp P \Lambda^{-1} X$ yields a $2N \times 2N$ symplectic matrix, i.e., of the block form
\begin{gather*}
M = \left(
\begin{matrix}
A & B\\
C & - \widetilde{A}
\end{matrix}
\right)
\end{gather*}
with $\widetilde{B} = B$ and $\widetilde{C} = C$, where for an $N \times N$ matrix $A$ we denote by $\widetilde{A}$ the transpose of $A$ along the minor diagonal. And for any such matrix $M$ we have
\begin{align*}
M &= \sum_{I,J \in \mathcal I} \tilde E_{IJ} M_{IJ} = \sum_{i,j = 1}^N \big( \big( \tilde E_{ij} - \tilde E_{-j,-i} \big) A_{ij} + \tilde E_{i,-j} C_{ij} - \tilde E_{-i,j} B_{ij} \big)\\
&= \sum_{i,j = 1}^N \bar E^{ij} A_{ji} + \sum_{\substack{i,j = 1\\ i \leq j}}^N \bar E^{-i,j} C_{ji} - \sum_{\substack{i,j = 1\\ i \leq j}}^N \bar E^{i,-j} B_{ji} = \sum_{(I,J) \in \mathcal I_2} \bar E^{IJ} M_{IJ}.
\end{align*}
Lastly, we clearly have $\det \Lambda = \prod\limits_{a=1}^M (\lambda - \lambda_a)$ and $\det Z = z^{2 \tau_0} \prod\limits_{i=1}^n (z - z_i)^{\tau_i}(z + z_i)^{\tau_i}$ from which the result now follows, using again the fact that $\det \trp A = \det A$ for any square matrix~$A$, as in the proof of Theorem~\ref{thm: cla bispec bos}.
\end{proof}

\begin{Remark} \label{rem: cyclo quantum}
Consider replacing $p$ by $\del$ in the $(M+2N) \times (M+2N)$ square matrix
\begin{gather*}
\bmx \lambda- \lambda_1 & & 0 &
p_N^1 & \dots & p_1^1 & x_1^1 & \dots & x_N^1 \\
& \ddots& &
\vdots & \ddots & \vdots &\vdots & \ddots & \vdots \\
0& & \lambda - \lambda_M & p_N^M & \dots & p_1^M & x_1^M & \dots & x_N^M\\
-x_N^1 & \dots & -x_N^M & & & & & & \\
\vdots & \ddots& \vdots & & & & & & \\
-x_1^1 & \dots & -x_1^M & & & & & & \\
 p_1^1 & \dots & p_1^M & & & Z & & & \\
\vdots & \ddots& \vdots & & & & & & \\
 p_N^1 & \dots & p_N^M & & & & & & \emx
\end{gather*}
used in the proof of Theorem \ref{thm: cla bispec bos cycl}. The resulting square matrix with non-commutative entries is not Manin since, for example, the entries of the first column are not mutually commuting. Consequently, we do not immediately obtain a quantum analogue of the classical relation in Theorem \ref{thm: cla bispec bos cycl}.

A related remark is that in the quantum case, higher Gaudin Hamiltonians for cyclotomic Gaudin models do exist but they are not in general given by a simple $\cdet$-type formula. See \cite{VY1,VY2} (and especially Remark 2.5 in \cite{VY1}).
\end{Remark}

\begin{Remark}
Note that we did not allow irregular singularities on the $\sp_{2N}$ side (appart from the double pole at infinity).

From the point of view of $(\gl_M, \gl_N)$-duality, the absence of irregular singularities in the $\sp_{2N}$-Gaudin model is controlled by the fact that the matrix
\begin{gather} \label{matrix at infinity}
\big( \pic_{\mathsf b}\big( \mathsf E^{+ (\infty)}_{ab \ul 1} \big) \big)_{a,b = 1}^M,
\end{gather}
representing the Casimir generators attached to infinity in the cyclotomic $\gl_M$-Gaudin model, is purely diagonal and in particular has no Jordan blocks, as in Lemma~\ref{lem: cyc pic pict}. Yet this is forced on us since the matrix~\eqref{matrix at infinity} is symmetric.

Alternatively, note that if one naively attempts to run the arguments above for the divisor~$\tilde \D$ in place of~$\bar \D$, one does not obtain a homomorphism $\sp_{2N}^{\tilde \D} \to \mathcal P_{\mathsf b}$. For example, Poisson brackets of the form $\big\{{-}\sum_{u} x^u_i x^{u+1}_j,\sum_{v} p^v_k p^{v+1}_l\big\}$ produce two sorts of terms: ``good'' terms like $\sum_u x^u_i p^{u+2}_l \delta_{jk}$, which respect the gradation of the Takiff algebra, but also ``bad'' terms like $\sum_u x^{u+1}_j p_l^{u+1}\delta_{ik}$, which do not.
\end{Remark}

\subsection{Example: Neumann model}

We end this section by considering the special case of Theorem \ref{thm: cla bispec bos cycl} when $N=1$ and $\mu = -1$.

Specifically, for the $\ZZ_2$-cyclotomic Gaudin model of Section~\ref{sec: Z2 cyclo} we take $n=0$ and $\tau_0 = 1$. The formal Lax matrix~\eqref{formal Lax glM cyclotomic} of the corresponding cyclotomic $\gl_M$-Gaudin model with effective divisor $\mathcal C = 2 \cdot 0 + 2 \cdot \infty$ then reduces to
\begin{gather} \label{Neumann formal Lax 1}
\tilde{\mc L}^\C(z) dz = \sum_{a, b=1}^M E_{ba} \otimes \left( \mathsf E^{+(\infty)}_{ab \ul 1} + \frac{\mathsf E^{-(0)}_{ab \ul 0}}{z} + \frac{\mathsf E^{+(0)}_{ab \ul 1}}{z^2} \right) dz.
\end{gather}

When $N=1$ in Section~\ref{sec: so2N} we have the canonical isomorphism $\sp_2 \simeq \mathfrak{sl}_2$ given by $\bar{\mathsf E}_{11} \mapsto - \mathsf H$, $\bar{\mathsf E}_{1,-1} \mapsto 2 \mathsf F$ and $\bar{\mathsf E}_{-1, 1} \mapsto 2 \mathsf E$. The dual basis elements are sent under this isomorphism to $\bar{\mathsf E}^{11} = \bar{\mathsf E}_{11} \mapsto - \mathsf H$, $\bar{\mathsf E}^{1,-1} = \ha \bar{\mathsf E}_{-1,1} \mapsto \mathsf E$ and $\bar{\mathsf E}^{-1,1} = \ha \bar{\mathsf E}_{1,-1} \mapsto \mathsf F$. The formal Lax matrix~\eqref{formal Lax so2N} of the $\mathfrak{sl}_2$-Gaudin model with effective divisor~\eqref{bar D} then becomes,
\begin{align} \label{Neumann formal Lax 2}
\mc L^{\bar\D}(\lambda) d\lambda= \Bigg(& H \otimes \mathsf H^{(\infty)} + 2 E \otimes \mathsf F^{(\infty)} + 2 F \otimes \mathsf E^{(\infty)} \notag\\
&{}+ \sum_{a=1}^M \frac{H \otimes \mathsf H^{(\lambda_a)} + 2 E \otimes \mathsf F^{(\lambda_a)} + 2 F \otimes \mathsf E^{(\lambda_a)}}{\lambda - \lambda_a} \Bigg) d\lambda,
\end{align}
where we have used the notation
\begin{gather*}
E \coloneqq \rho(\mathsf E) = \left(
\begin{matrix}
0 & 1\\ 0 & 0
\end{matrix}
\right), \qquad
F \coloneqq \rho(\mathsf F) = \left(
\begin{matrix}
0 & 0\\ 1 & 0
\end{matrix}
\right), \qquad
H \coloneqq \rho(\mathsf H) = \left(
\begin{matrix}
1 & 0\\ 0 & -1
\end{matrix}
\right).
\end{gather*}

The Poisson algebra $\mathcal P_{\mathsf b}$ in the present context is simply $\CC[x_a, p_a]_{a,b=1}^M$ where we have dropped the subscript $1$ from the canonical variables by defining $x_a \coloneqq x^a_1$ and $p_a \coloneqq p^a_1$. In terms of this notation, the representation $\pi_{\mathsf b} \colon \gl^{\mathcal C}_M \to \mathcal P_{\mathsf b}$ from Theorem~\ref{thm: cla bispec bos cycl} reads
\begin{gather*}
\pic_{\mathsf b}\big(\mathsf E^{+(\infty)}_{ab \ul 1}\big) = \lambda_a \delta_{ab}, \qquad
\pic_{\mathsf b}\big(\mathsf E^{-(0)}_{ab \ul 0}\big) = x_a p_b - x_b p_a, \qquad
\pic_{\mathsf b}\big(\mathsf E^{+(0)}_{ab \ul 1}\big) = - x_a x_b,
\end{gather*}
recalling that $\mu = -1$. Correspondingly, the map $\picb_{\mathsf b}\colon \sp^{\mathcal C}_2 \to \mathcal P_{\mathsf b}$ takes the form
\begin{alignat*}{4}
& \picb_{\mathsf b}\big(\mathsf E^{(\infty)}\big) = \ha, \qquad &&
\picb_{\mathsf b}\big(\mathsf F^{(\infty)}\big) = 0, \qquad &&
\picb_{\mathsf b}\big(\mathsf H^{(\infty)}\big)= 0,& \\
&\picb_{\mathsf b}\big(\mathsf E^{(\lambda_a)}\big)= \ha p_a^2, \qquad &&
\picb_{\mathsf b}\big(\mathsf F^{(\lambda_a)}\big) = - \ha x_a^2, \qquad &&
\picb_{\mathsf b}\big(\mathsf H^{(\lambda_a)}\big) = x_a p_a.&
\end{alignat*}

Applying the first representation $\pic_{\mathsf b}$ to the formal Lax matrix \eqref{Neumann formal Lax 1} we find
\begin{align*}
\tilde L(z)dz &\coloneqq \pic_{\mathsf b} \big( \tilde{\mc L}^\C(z) \big) dz\\
&\, = \left( \sum_{a=1}^M \lambda_a E_{aa} - z^{-1} \sum_{a,b=1}^M (x_a p_b - x_b p_a) E_{ab} - z^{-2} \sum_{a,b=1}^M x_a x_b E_{ab} \right) dz.
\end{align*}
If we introduce variables $\omega_a$, $a = 1, \ldots, M$ such that $\omega_a^2 = \lambda_a$ then the above coincides with the $M \times M$ Lax matrix of the Neumann model, with Hamiltonian
\begin{gather*}
H = \frac{1}{4} \sum_{\substack{a,b = 1\\ a \neq b}}^M (x_a p_b - x_b p_a)^2 + \frac{1}{2} \sum_{a=1}^M \omega_a^2 x_a^2,
\end{gather*}
describing the motion of a particle constrained to the sphere $\sum\limits_{a=1}^M x_a^2 = 1$ in $\mathbb{R}^M$ and subject to harmonic forces with frequency $\omega_a$ along the $a^{\rm th}$ axis. On the other hand, applying~$\picb_{\mathsf b}$ to the formal Lax matrix~\eqref{Neumann formal Lax 2} yields
\begin{gather*}
L(\lambda) d\lambda \coloneqq \picb_{\mathsf b} \big( \mc L^{\bar \D}(\lambda) \big) d\lambda = 2 \left(
\begin{matrix}
\sum\limits_{a=1}^M \frac{x_a p_a}{\lambda - \lambda_a} & \sum\limits_{a=1}^M \frac{- x_a^2}{\lambda - \lambda_a}\\
1 + \sum\limits_{a=1}^M \frac{p_a^2}{\lambda - \lambda_a} & - \sum\limits_{a=1}^M \frac{x_a p_a}{\lambda - \lambda_a}
\end{matrix}
\right) d\lambda,
\end{gather*}
which coincides with the expression for the $2 \times 2$ Lax matrix of the same model. The statement of Theorem~\ref{thm: cla bispec bos cycl} corresponds to the well known relation between the above two Lax formulations of the Neumann model (see, e.g., \cite[Section~12]{Discrete})
\begin{gather*}
z^2 \det\big( \lambda \mathbf{1}_{M \times M} - \tilde L(z) \big) = \prod_{a=1}^M (\lambda - \lambda_a) \det\big( z \mathbf{1}_{2 \times 2} - L(\lambda) \big).
\end{gather*}

\pdfbookmark[1]{References}{ref}
\LastPageEnding

\end{document}